\documentclass[10pt]{amsart}
\usepackage{amsmath}
\usepackage{amsthm}
\usepackage{amscd}
\usepackage{amssymb}
\usepackage{amsfonts}
\usepackage{graphics}
\usepackage{color}
\usepackage[dvipsnames]{xcolor}
\usepackage{comment}
\usepackage[normalem]{ulem}
\usepackage{orcidlink}
\date{}

\newlength{\defbaselineskip }
\long\def\salta#1{\relax}

 \theoremstyle{plain}
\newtheorem{theorem}{Theorem}[section]
\newtheorem{proposition}[theorem]{Proposition}
\newtheorem{lemma}[theorem]{Lemma}

\theoremstyle{definition}

\newtheorem{remark}[theorem]{Remark}

\newcommand{\re}{\mathbb{R}}

\newcommand{\bu}{\bar{u}}
\newcommand{\bv}{\bar{v}}
\newcommand{\bz}{\bar{z}}

\def\R{\mathbb{R}}

\def\t1p0{T^{1,p}_{0}(\Omega)}

\def\m2{M^{\frac{N(p-1)}{N-1}}(\Omega)}

\def\sobq{W^{1,q}_{0}(\Omega)}
\def\sob{W^{1,p}_{0}(\Omega)}

\def\into{\int_{\Omega}}

\def\w-1p'{W^{-1,p'}(\Omega)}
\def\pw-1p'u{L^{p'}(0,1;W^{-1,p'}(\Omega))}

\def\lp'n{(L^{p'}(\Omega))^{N}}

\numberwithin{equation}{section}

\title[
On the de Thélin eigenvalue problem and Landesman-Lazer conditions]{On the de Thélin eigenvalue problem and Landesman-Lazer conditions for quasilinear systems}

\author[D. Arcoya]{David Arcoya \orcidlink{https://orcid.org/0000-0002-7284-2413} }
\address{Departamento de Análisis Matemático, Campus Fuentenueva S/N, Universidad de Granada, 18071 Granada, Spain}
\email{darcoya@ugr.es}
\author[N. Borgia]{Natalino Borgia \orcidlink{https://orcid.org/0009-0004-4598-9542} }
\address{Dipartimento di  Matematica  \\ Universit\`{a} degli Studi di Bari Aldo Moro \\ Via Orabona 4\\ 70125 Bari, Italy}
\email{natalino.borgia@uniba.it}
\author[S. Cingolani]{Silvia Cingolani \orcidlink{https://orcid.org/0000-0002-3680-9106}}
\address{Dipartimento di  Matematica  \\ Universit\`{a} degli Studi di Bari Aldo Moro \\ Via Orabona 4\\ 70125 Bari, Italy}
\email{silvia.cingolani@uniba.it}

\begin{document}

\begin{abstract}
In this paper we prove that  the smallest eigenvalue $\lambda_1$ of the eigenvalue problem for  a quasilinear elliptic systems introduced by de Thélin in \cite{DT},  is not only simple (in a suitable sense), but also isolated.
 Moreover, we characterize  variationally a sequence $\{\lambda_k\}_k$ of eigenvalues, 
taking into account a suitable deformation lemma for $C^1$ submanifolds proved in \cite{BON}. Furthermore we prove the existence of a weak solution for a quasilinear elliptic systems in resonance around $\lambda_1$, under new sufficient Landesman-Lazer type conditions, 
extending the results by Arcoya and Orsina \cite{AO}.
\end{abstract}

\keywords{Quasilinear Elliptic Systems, Eigenvalue problem, Resonant problems}

\subjclass[2000]{35B34,35J62,35J92,47J30,58E05}

\maketitle

\section{INTRODUCTION}

\medskip
\noindent
In this paper, we study the following eigenvalue problem for quasilinear elliptic systems, originally introduced by de Thélin \cite{DT}:
\begin{equation}\label{EigenDeThélin0}
	\begin{cases}
		\begin{array}{ll}
			- \Delta_p u = \lambda |u|^{\alpha-1} |v|^{\beta+1} u
			&  \text{ in }\Omega,
			\medskip \\
			- \Delta_q v  = \lambda  |u|^{\alpha+1} |v|^{\beta - 1}v 
			&  \text{ in } \Omega, \medskip \\
			u=v=0  & \text{ on }  \partial\Omega,
		\end{array}
	\end{cases}
\end{equation}
where $\Omega$ is a bounded domain of $\mathbb{R}^N$ with smooth boundary, $N \geq 2$, $\lambda \in \R$, and $p, q,\alpha, \beta$ are real numbers satisfying:
\begin{equation}\label{condizionepqalphabeta}
	\displaystyle  1 < p < N , \quad 1 < q < N, \quad \alpha >0, \quad \beta >0, \quad \text{and}  \quad \frac{\alpha + 1}{p} + \frac{\beta + 1}{q} = 1.
\end{equation}
Here  $\Delta_r u := \text{ div } ( \left| \nabla u \right|^{r-2} \nabla u )$ with $r>1$, and the space setting is the product space $X:= W_0^{1,p}(\Omega)\times W_0^{1,q}(\Omega)$ endowed with the norm
\begin{equation}\label{normaX}
\|z\|= \|u\|_{1,p} + \|v\|_{1,q}, \quad z=(u,v)\in X,
\end{equation}	
\noindent
where  $\| \cdot \|_{1,r}$ denotes the usual gradient norm in $W^{1,r}_0(\Omega)$.\\ We will also denote with $\| \cdot \|_{r}$ the usual norm in $L^r(\Omega)$. A real number $\lambda$ is said to be an eigenvalue of \eqref{EigenDeThélin0} if there exists $(\varphi, \psi) \in X$, $(\varphi, \psi) \neq (0,0)$ such that
\begin{align*}
\displaystyle
& (\alpha+1) \int_{\Omega} |\nabla \varphi |^{p-2} \nabla \varphi \cdot \nabla u \, dx 
+ (\beta+1) \int_{\Omega} |\nabla \psi|^{q-2} \nabla \psi \cdot \nabla v \, dx \\
= &  (\alpha + 1) \into \left| \varphi \right|^{\alpha-1}  \left| \psi \right|^{\beta+1} \varphi u \, dx 
+ (\beta + 1) \into \left| \varphi \right|^{\alpha+1}  \left| \psi \right|^{\beta-1} \psi v  \, dx
\end{align*}
for any $(u, v) \in X$. The nontrivial weak solution $(\varphi, \psi)$  is then called eigenfunction associated to the eigenvalue $\lambda$.\\
Notice that if $(\varphi,\psi)$ is a nontrivial (i.e. $(\varphi, \psi) \neq (0,0)$) eigenfunction of \eqref{EigenDeThélin0} associated to an eigenvalue $\lambda \in \mathbb{R}$, then it cannot be semitrivial (i.e. $ \varphi \neq 0 \neq \psi$).\\
Let us introduce the functionals $\Phi : X \to \re$ and $\Psi:X \to \re$ defined as
\begin{align}\label{funzionalePhi}
\displaystyle
\Phi (u,v) =  \frac{\alpha+1}{p} \int_{\Omega} |\nabla u|^p \, dx + \frac{\beta+1}{q}\int_{\Omega} |\nabla v|^q \, dx
\end{align}
and
\begin{align}\label{funzionalePsi}
	\Psi(u,v) =  \into \left| u \right|^{\alpha+1}  \left| v \right|^{\beta+1} \, dx, 
\end{align}
for any $(u,v) \in X$, and let us denote with $\Sigma$  the  $C^1$ manifold
\begin{align*}
\displaystyle 
\Sigma:=\Biggl\{ z=(u,v) \in X \, : \,  \Psi(z)=1   \Biggr\}.
\end{align*}

In \cite[Theorem 1]{DT} de Thélin proved that problem \eqref{EigenDeThélin0} admits a smallest eigenvalue $\lambda_1>0$ defined as
\begin{equation}\label{definizioneprimoautovalore}
 \displaystyle \lambda_1 := \inf_{z \in \Sigma} \Phi(z).
\end{equation}

In \cite[Theorem 2]{DT}  he also proved, by means of Tolksdorf regularity results \cite{tolksdorf1983} and Vazquez's maximum principle \cite{VAZ},  that there exists a nontrivial eigenfunction  $(\varphi_0,\psi_0)$ associated to $\lambda_1$ such that $\varphi_0 \in C^{1,\eta}(\overline{\Omega})$, $\psi_0 \in C^{1,\xi}(\overline{\Omega})$ and $\varphi_0 > 0,$ $\psi_0 > 0$ in $\Omega$. In addition, by using results of Diaz and Saa \cite{DIAZSAA} and Anane \cite{AL}, in \cite[Theorem 3]{DT} it was proved that if $(\varphi_0,\psi_0)$ is an eigenfunction associated to $\lambda_1$  satisfying $\varphi_0 > 0,$ $\psi_0> 0$ in $\Omega$ and $(\varphi_0,\psi_0) \in \Sigma$, then is unique.

In this paper we prove that the eigenvalue $\lambda_1$ is isolated, and it is simple in the sense stated below.\\
Firstly, if $\lambda$ is an eigenvalue with a nontrivial eigenfunction $(\varphi, \psi)$, by using the relation $(\alpha+1)/p + (\beta+1)/q=1$ it turns out that 
$$ \displaystyle E:= \bigg\{ ( \left|  \theta \right|^{\frac{1}{p}} \varphi ,  \left|  \theta \right|^{\frac{1}{q}} \psi ) \, \text{sgn} (\theta) \, : \, \theta \in \mathbb{R} \bigg\} \,  \bigcup \, \bigg\{ (- \left|  \theta \right|^{\frac{1}{p}} \varphi ,  \left|  \theta \right|^{\frac{1}{q}} \psi ) \, \text{sgn} (\theta) \, : \, \theta \in \mathbb{R} \bigg\} $$
is a set of eigenfunctions associated to $\lambda$ and the associated eigenspace is not a linear subspace. However, if $(\bar{\varphi},\bar{\psi)}$ is another eigenfunction associated with the same eigenvalue $\lambda$, it could  be the case that  $(\bar{\varphi},\bar{\psi)} \not \in E$. \cite[Theorem~3]{DT} proves that this possibility is excluded for the eigenvalue $\lambda_1$ (simplicity of  $\lambda_1$).

\begin{theorem}\label{simplicityINTRODUZIONE}
The set of all eigenfunctions associated to $\lambda_1$ is given by
$$ \displaystyle E_1:= \bigg\{ ( \left|  \theta \right|^{\frac{1}{p}} \varphi_0 ,  \left|  \theta \right|^{\frac{1}{q}} \psi_0 ) \, \text{sgn} (\theta) \, : \, \theta \in \mathbb{R} \bigg\} \,  \bigcup \, \bigg\{ (- \left|  \theta \right|^{\frac{1}{p}} \varphi_0 ,  \left|  \theta \right|^{\frac{1}{q}} \psi_0 ) \, \text{sgn} (\theta) \, : \, \theta \in \mathbb{R} \bigg\} $$
where $(\varphi_0,\psi_0)$ is the unique nontrivial eigenfunction satisfying $(\varphi_0,\psi_0) \in \Sigma$,    $\varphi_0 >0$ and $ \psi_0>0$  in $\Omega$.
\end{theorem}

Previous result and Picone's identity (see \cite[Theorem 1]{AH} and \cite[Lemma 24]{AG}), allow us to prove that if $(\varphi,\psi)$ is an eigenfunction of \eqref{EigenDeThélin0} associated to the eigenvalue $\lambda$ with $\lambda > \lambda_1$, then both $\varphi$ and $\psi$ change sign (see Proposition \ref{funzionicambianosegno}).\\
By means of this result, we are able to prove that:

\begin{theorem}\label{lambda1isolatoINTRODUZIONE}
The eigenvalue $\lambda_1$ is isolated, that is there exists $a > \lambda_1$  such that $\lambda_1$ is the only eigenvalue of \eqref{EigenDeThélin0} in  $[0,a]$. 
\end{theorem}

Moreover, inspired by \cite{DR}, we also construct a sequence $\{\lambda_k\}_k$ of eigenvalues of \eqref{EigenDeThélin0} with a suitable variational characterization involving the unit sphere $S^{k-1}$ of $\mathbb{R}^k$ for any $k \geq 1$. Such a sequence will be construct on the $C^1$ manifold 

\begin{align*}
\displaystyle 
\mathcal{M}:=\Biggl\{ z=(u,v) \in X \, : \,  \Phi(z)=1   \Biggr\}.
\end{align*}

\medskip
\noindent
In particular, setting

$$
\mathcal{M}_k := \{ A=\alpha(S^{k-1}) \subset \mathcal{M}  \ | \,  \alpha: S^{k-1} \to \mathcal{M} \text{ is odd and continuous}  \},
$$

\medskip
\noindent
and

\begin{equation*}
	c_k := \sup_{A \in \mathcal{M}_k} \min_{(u,v) \in A} \Psi(u,v),
\end{equation*}

\medskip
\noindent
we prove in Proposition \ref{prop 2.9} that the sequence $\{\lambda_k\}_k \subset \mathbb{R}$ where

\begin{equation}\label{lambdakINTRODUZIONE}
	\lambda_k := \frac{1}{c_k},
\end{equation}

\medskip
\noindent
is a sequence of eigenvalues of \eqref{EigenDeThélin0}, and $\lambda_1$ is exactly the first eigenvalue defined in \eqref{definizioneprimoautovalore} (see Remark \ref{definizionicoincidono} in Section 4).

In order to construct a sequence of eigenvalues of \eqref{EigenDeThélin0}, we notice that the manifold $\mathcal{M}$ is not of class $C^{1,1}$ when $ 1 <p < 2$ or $1<q<2$. We will apply a deformation lemma for $C^1$ submanifolds of a Banach space (see Theorem~\ref{BON}) based in \cite[Theorem 2.5]{BON}. This theorem requires a coupled Palais-Smale condition.

As in the single scalar  eigenvalue problem $-\Delta_p u = \eta |u|^{p-2}u$ it is unknown whether the set of the eigenvalues described by \eqref{lambdakINTRODUZIONE} contains all the eigenvalues of problem \eqref{EigenDeThélin0}.

In the context of quasilinear elliptic systems, we also recall that there are several ways to define an eigenvalue problem (see for instance \cite{SZ} for an other type of eigenvalue problem different from \eqref{EigenDeThélin0}, and also \cite{boccardodefiguerido} for a more general class of eigenvalue problems involving $(p,q)$-homogeneous functions).

Then we consider the eigenvalue problem \eqref{EigenDeThélin0}, by proving that $\lambda_1$ is isolated and simple, and admits the sequence $\{\lambda_k \}_k$ defined in \eqref{lambdakINTRODUZIONE} of eigenvalues, with the idea of proving the same results for other eigenvalue problems in a forthcoming paper.

Here we also concern with  the existence of a weak solution for a quasilinear elliptic systems in resonance around $\lambda_1$, under new sufficient Landesman-Lazer type conditions, 
extending the results by Arcoya and Orsina \cite{AO}.\\
We then consider the following system:
\begin{equation}\label{Syst0}
	\begin{cases}
		\begin{array}{ll}
			-\Delta_p u = \lambda_1 |u|^{\alpha-1} |v|^{\beta+1} u + \frac{1}{\alpha + 1} [F_s(x,u,v) - h_1(x) ]& x\in\Omega,
			\bigskip
			\\
			-\Delta_q v= \lambda_1  |u|^{\alpha+1} |v|^{\beta - 1}v  + \frac{1}{\beta + 1} [ F_t(x,u,v) - h_2(x)] & x\in \Omega,
			\bigskip
			\\
			u=v=0  & x\in \partial\Omega.
		\end{array}
	\end{cases}
\end{equation}

\medskip
\noindent
Here the nonlinearity  $F:\Omega \times \mathbb{R}^2 \to \mathbb{R}$ is a $C^1$-Carathéodory function (that is, $F(x, s, t)$ is measurable with respect to $ x \in \Omega$ for every $(s,t) \in \mathbb{R}^2$, and of class $C^1$ with respect to $(s,t) \in \mathbb{R}^2$ for a.e. $ x \in \Omega$), while $h_1 \in L^{p'}(\Omega)$ and  $h_2 \in L^{q'}(\Omega)$, where $p':= p/(p-1)$ and $q':=q/(q-1)$.
\noindent
Inspired by \cite{AO}, our aim is to obtain a weak solution of \eqref{Syst0} in the spirit of the work \cite{LANDLAZ} of Landesman and Lazer (see also \cite{ALP,AMBROSETTIMANCINI,HESS} for semilinear equations, and \cite{ANGOS,AO,BDK,DR} for the quasilinear equations).

\medskip
\noindent
We assume that $F$ satisfies also the following conditions:\\

\begin{itemize}
\item[\textbf{(F1)}] $F(x,0,0) \in L^1(\Omega)$ and there exists $M>0$ such that
$$ \displaystyle |F_s(x,s,t)| \leq M \quad \text{and} \quad |F_t(x,s,t)| \leq M \qquad \text{ for all } (x,s,t) \in \Omega \times \mathbb{R}^2.$$\\

\item[\textbf{(F2)}] For almost every $x \in \Omega$, there exist
\begin{align*}
\displaystyle 
& F_s^{++}(x)= \lim_{\substack{s \to + \infty \\ t  \to  + \infty}} F_t(x,s,t), & F_t^{++}(x)= \lim_{\substack{s \to + \infty \\ t  \to  + \infty}} F_s(x,s,t), \\
& F_s^{+-}(x)= \lim_{\substack{s \to + \infty \\ t  \to  - \infty}} F_t(x,s,t), & F_t^{+-}(x)= \lim_{\substack{s \to + \infty \\ t  \to  - \infty}} F_s(x,s,t), \\
& F_s^{-+}(x)= \lim_{\substack{s \to - \infty \\ t  \to  + \infty}} F_t(x,s,t), & F_t^{-+}(x)= \lim_{\substack{s \to - \infty \\ t  \to  + \infty}} F_t(x,s,t), \\
& F_s^{--}(x)= \lim_{\substack{s \to - \infty \\ t  \to  - \infty}} F_t(x,s,t), & F_t^{--}(x)= \lim_{\substack{s \to - \infty \\ t  \to  - \infty}} F_t(x,s,t).
\end{align*}
\end{itemize}

\medskip
\noindent
Moreover, setting

$$ \displaystyle (\varphi_1,\psi_1):=\left( \frac{\varphi_0}{ \left( \lVert \varphi_0 \rVert_{1,p}^p + \lVert \psi_0 \rVert_{1,q}^q \right)^{\frac{1}{p}}}  ,\frac{\psi_0}{ \left( \lVert \varphi_0 \rVert_{1,p}^p + \lVert \psi_0 \rVert_{1,q}^q \right)^{\frac{1}{q}}} \right),$$

\medskip
\noindent
we infer that  $(\varphi_1,\psi_1)$ is an eigenfunction associated to $\lambda_1$ with 

$$\displaystyle \lVert \varphi_1 \rVert^p_{1,p} + \lVert \psi_1 \rVert^q_{1,q}=1.$$

\medskip
\noindent
In particular, by $(p,q)$-homogeneity we can write $E_1$ given by Theorem \ref{simplicityINTRODUZIONE} also as

\begin{equation}\label{InsiemeE1riscrittoINTRODUZIONE}
 \displaystyle E_1= \bigg\{ ( \left|  \theta \right|^{\frac{1}{p}} \varphi_1 ,  \left|  \theta \right|^{\frac{1}{q}} \psi_1 ) \, \text{sgn} (\theta) \, : \, \theta \in \mathbb{R} \bigg\} \,  \bigcup \, \bigg\{ (- \left|  \theta \right|^{\frac{1}{p}} \varphi_1 ,  \left|  \theta \right|^{\frac{1}{q}} \psi_1 ) \, \text{sgn} (\theta) \, : \, \theta \in \mathbb{R} \bigg\}.
\end{equation}

\noindent
We then introduce the following sufficient Landesman-Lazer type conditions, involving $p,q,F,h_1,h_2$ and the eigenfunction $(\varphi_1,\psi_1)$ (we omit $dx$ for clarity of reading):\\
if $p < q$,  either
\begin{equation}\label{LLpminoreq1}
\begin{cases}
\displaystyle \int_{\Omega} F_s^{++}\varphi_1 \,  < \int_{\Omega} h_1 \varphi_1 \,  < \int_{\Omega} F_s^{--} \varphi_1 \bigskip\\
\text{and} \bigskip\\
\displaystyle
 \int_{\Omega} F_s^{+-}\varphi_1 \,   < \int_{\Omega} h_1 \varphi_1 \,  < \int_{\Omega} F_s^{-+} \varphi_1,
 \end{cases}
\end{equation}
or
\begin{equation}\label{LLpminoreq2}
\begin{cases}
\displaystyle
 \int_{\Omega} F_s^{--}\varphi_1 \,   < \int_{\Omega} h_1 \varphi_1 \,  < \int_{\Omega} F_s^{++} \varphi_1 \,  \bigskip\\
\text{and} \bigskip\\
\displaystyle
 \int_{\Omega} F_s^{-+}\varphi_1 \,  < \int_{\Omega} h_1 \varphi_1 \,  < \int_{\Omega} F_s^{+-} \varphi_1;
 \end{cases}
\end{equation}
if $p=q$,  either
\begin{equation}\label{LLpp1}
\begin{cases}
\displaystyle
 \into F_s^{++}\varphi_1  + F_t^{++}\psi_1 \,  <  \into h_1 \varphi_1 +  h_2 \psi_1  \,   <  \into F_s^{--} \varphi_1 +  F_t^{--} \psi_1 \, \bigskip\\
\text{and} \bigskip\\
\displaystyle
    \into F_s^{+-} \varphi_1 - F_t^{+-} \psi_1 \, <  \into h_1 \varphi_1  - h_2 \psi_1 \,  < \into F_s^{-+} \varphi_1- F_t^{-+} \psi_1,
    \end{cases}
\end{equation}
or
\begin{equation}\label{LLpp2}
\begin{cases}
\displaystyle
    \into F_s^{--} \varphi_1 + F_t^{--} \psi_1 \, <  \into h_1 \varphi_1  + h_2 \psi_1 \, < \into F_s^{++} \varphi_1+ F_t^{++} \psi_1 \,  \bigskip\\
\text{and} \bigskip\\
\displaystyle
    \into F_s^{-+} \varphi_1 - F_t^{-+} \psi_1 \, <  \into h_1 \varphi_1  - h_2 \psi_1 \,  < \into F_s^{+-} \varphi_1 - F_t^{+-} \psi_1;
\end{cases}
\end{equation}
if $p>q$,  either
\begin{equation}\label{LLpmaggioreq1}
\begin{cases}
\displaystyle
 \int_{\Omega} F_t^{++} \psi_1 \,  < \int_{\Omega} h_2 \psi_1 \,  < \int_{\Omega} F_t^{--} \psi_1  \bigskip \\
\text{and} \bigskip \\
\displaystyle
 \int_{\Omega} F_t^{-+} \psi_1\,  < \int_{\Omega} h_2 \psi_1\,  < \int_{\Omega} F_t^{+-} \psi_1, 
\end{cases} 
\end{equation}
or
\begin{equation}\label{LLpmaggioreq2}
\begin{cases}
\displaystyle
 \int_{\Omega} F_t^{--} \psi_1\,  < \int_{\Omega} h_2 \psi_1\,  < \int_{\Omega} F_t^{++} \psi_1  \bigskip \\
\text{and} \bigskip\\
\displaystyle
 \int_{\Omega} F_t^{+-} \psi_1\, < \int_{\Omega} h_2 \psi_1\,  < \int_{\Omega} F_t^{-+} \psi_1.
 \end{cases} 
\end{equation}

\noindent
Our main result is the following.

\begin{theorem}\label{mainresult}
Let $F:\Omega \times \mathbb{R}^2 \to \mathbb{R}$ be a $C^1$-Carathéodory function satisfying \textbf{(F1)} and \textbf{(F2)}. Consider $h_1 \in L^{p'}(\Omega)$ and  $h_2 \in L^{q'}(\Omega)$. If we assume that 
\begin{itemize}

\item[•]  in the case $p<q$,  either condition \eqref{LLpminoreq1} or  condition \eqref{LLpminoreq2} holds true;

\item[•]  in the case $p=q$,  either condition \eqref{LLpp1} or  condition \eqref{LLpp2} holds true;

\item[•]  in the case $p>q$,  either condition \eqref{LLpmaggioreq1} or  condition \eqref{LLpmaggioreq2} holds true;
\end{itemize}
then there exists at least a weak solution $(u,v) \in X$ for problem \eqref{Syst0}. 
\end{theorem}

In the proof of Theorem \ref{mainresult}, we show that in any case the Euler functional $J$ associated to \eqref{Syst0} satisfies the compactness condition of Palais-Smale.

In addition, under the assumptions \eqref{LLpminoreq1} if $p<q$, \eqref{LLpp1} if $p=q$ and \eqref{LLpmaggioreq1} if $p>q$, the functional $J$ is coercive on $X$, and so \eqref{Syst0} has a solution. 

On the other hand, that is under the assumptions  \eqref{LLpminoreq2} if $p<q$, \eqref{LLpp2} if $p=q$ and \eqref{LLpmaggioreq2} if $p>q$, the extension of the results contained in \cite{AO} to problem \eqref{Syst0} is not straightforward, since the set $E_1$  given by \eqref{InsiemeE1riscrittoINTRODUZIONE} is not a vector subspace of $X$.

Indeed, in \cite{AO} it was possible to apply the saddle point theorem of Rabinowitz (see  \cite{RAB}), exploiting the fact that whenever $W_0^{1,p}(\Omega)$ is splitted as the direct sum of $\langle \varphi \rangle$ and $V$, where $\varphi$ is a nontrivial eigenfunction associated to the first eigenvalue $\mu_1$ of $-\Delta_p$ on $\Omega$ with respect to Dirichlet boundary condition, then from standard argument it follows that there exists $\bar{\lambda} > \mu_1$ such that $\into \left| \nabla u \right|^p \, dx \geq \bar{\lambda} \into |u|^p \, dx $ for any $u \in V$. In particular, this argument requires just the fact the eigenspace $\langle  \varphi \rangle$ associated to $\mu_1$ is a finite-dimensional vector subspace of $W_0^{1,p}(\Omega)$, without using any other information on the spectrum of the $p$-Laplacian, even the fact that $\mu_1$ is isolated.

Here we cannot apply the same argument, since the set $E_1$ given by \eqref{InsiemeE1riscrittoINTRODUZIONE} is never a vector subspace of $X$, even in the case $p=q$, in which case we have the union of two vector subspaces. We overcome this difficulty  by using the fact that $\lambda_1$ is isolated and the variational characterization of $\lambda_2$, therefore we are able to prove that two suitable sets satisfy the geometry of saddle point and are linked in a way that allows the application of standard minimax theorems (see \cite[Theorem 3.4]{STRUWE}).

\section{SIMPLICITY OF $\lambda_1$}

\medskip
\noindent
The functionals $\Phi : X \to \re$ and $\Psi:X \to \re$ defined in \eqref{funzionalePhi} and \eqref{funzionalePsi} are of class $C^1$ and for any $(u_0,v_0), (u,v) \in X$ we have
\begin{align*}
\displaystyle \langle \Phi '(u_0,v_0), (u,v)  \rangle
&  = (\alpha+1) \int_{\Omega} |\nabla u_0|^{p-2} \nabla u_0 \cdot \nabla u \, dx \\
& + (\beta+1) \int_{\Omega} |\nabla v_0|^{q-2} \nabla v_0 \cdot \nabla v \, dx
\end{align*}
and
\begin{align*}
\displaystyle \langle \Psi '(u_0,v_0), (u,v)  \rangle 
& = (\alpha + 1) \into \left| u_0 \right|^{\alpha-1}  \left| v_0 \right|^{\beta+1} u_0 u \, dx \\
& + (\beta + 1) \into \left| u_0 \right|^{\alpha+1}  \left| v_0 \right|^{\beta-1} v_0 v \, dx.
\end{align*}
By \eqref{condizionepqalphabeta} we get that $\Phi$ and $\Psi$ are $(p,q)$-homogeneous functionals, i.e.
\begin{equation*}
	\Phi(\theta^{\frac 1 p}u,\theta^{\frac 1 q}v)=\theta\Phi(u,v), \quad \Psi(\theta^{\frac 1 p}u,\theta^{\frac 1 q}v)=\theta\Psi(u,v) \qquad \forall \theta>0,\ (u,v)\in X.
\end{equation*}

Notice that $\Psi(u,v)=0$ whenever $u(x)v(x)=0$ a.e. $x \in \Omega$. Hence, denoting with $\tilde{X}$ the set of $(u,v) \in X$ such that the product $uv$ is different from the zero function on $\Omega$, i.e.  
$$\tilde{X}:= \{ (u,v) \in X \, : \, u v \not \equiv 0\},$$
we can characterize $\displaystyle \lambda_1 := \inf_{z \in \Sigma} \Phi(z)$ also as
$$ \displaystyle \lambda_1 = \inf_{z \in \tilde{X}} \frac{\Phi(z)}{\Psi(z)}.$$

In \cite[Theorem 2]{DT} de Thélin showed through Tolksdorf regularity results \cite{tolksdorf1983} that every nontrivial eigenfunction  $(\varphi,\psi)$ satisfies $\varphi \in C^{1,\eta}(\overline{\Omega})$, $\psi \in C^{1,\xi}(\overline{\Omega})$. In addition, through Vazquez's maximum principle \cite{VAZ}, for every nontrivial eigenfunction $(\varphi_0, \psi_0)$ associated to $\lambda_1$ with   $\varphi_0 \geq 0$  and $\psi_0 \geq 0$, we have  necessarily $\varphi_0 > 0,$ $\psi_0 > 0$ in $\Omega$ and  $\frac{\partial \varphi_0}{\partial \nu}, \frac{\partial \psi_0}{\partial \nu} < 0$ on $\partial \Omega$, where $\nu$ denotes the outward normal to $\partial \Omega$. This result is easily extended to any eigenvalue $\lambda$ of \eqref{EigenDeThélin0}.

\medskip
\noindent
{\mbox {\it Proof of Theorem~\ref{simplicityINTRODUZIONE}.~}  It is contained in \cite{DT}, but for the reader's convenience, we provide a sketch here.
First of all, we observe that any $(\varphi, \psi) \in E_1$ is an eigenfunction of \eqref{EigenDeThélin0}  associated to the eigenvalue $\lambda_1$. Conversely, let $(\varphi,\psi)$ be an eigenfunction associated to $\lambda_1$. Since
$$ \displaystyle \frac{ \Phi (|\varphi|,|\psi|) }{\Psi (|\varphi|,|\psi|)}=\frac{ \Phi (\varphi,\psi) }{\Psi (\varphi,\psi)} = \lambda_1,   $$
we deduce that $(\left|\varphi\right|, \left|\psi \right|)$ is an eigenfunction associated to $\lambda_1$.\\ We know that $|\varphi| \in C^{1,\eta}(\overline{\Omega})$, $|\psi| \in C^{1,\xi}(\overline{\Omega})$ and $|\varphi| > 0,$ $|\psi| > 0$ in $\Omega$. Therefore, since $\Omega$ is a connected set, both $\varphi$ and $\psi$ cannot change sign.\\
Let us assume for example $\varphi > 0$ and $\psi > 0$.
Denoting with
$$ \displaystyle \theta := \into \left| \varphi \right|^{\alpha+1}  \left| \psi \right|^{\beta+1} \, dx, $$
\noindent
we have that 
$$ \displaystyle (\bar{\varphi}, \bar{\psi}):= \left(  \frac{\varphi}{\theta^{\frac{1}{p}}} ,  \frac{\psi}{\theta^{\frac{1}{q}}} \right) $$
is an eigenfunction associated to $\lambda_1$ such that $\bar{\varphi} > 0$, $\bar{\psi} > 0$ and $(\bar{\varphi}, \bar{\psi}) \in \Sigma$. By \cite[Theorem 3]{DT} 
there is a unique eigenfunction $(\varphi_0,\psi_0)$ associated to $\lambda_1$  satisfying $(\varphi_0,\psi_0) \in \Sigma$,    $\varphi_0 >0$ and $ \psi_0>0$  in $\Omega$ and
we deduce $(\bar{\varphi}, \bar{\psi})= (\varphi_0,\psi_0)$, hence $(\varphi,\psi) \in E_1$. We have proved the result assuming $\varphi > 0$ and $\psi > 0$.

The same result holds if we argue as before on the vector $(\varphi, - \psi)$ if $\varphi > 0$ and $\psi <0$, $(-\varphi,  \psi)$ if $\varphi <  0$ and $\psi > 0$, and $(-\varphi, - \psi)$ if $\varphi < 0$ and $\psi < 0$.

\qed

\section{ISOLATION OF $\lambda_1$}

\medskip
\noindent
To prove that $\lambda_1$ is isolated, we follow the ideas of \cite[Theorem 1]{AH}  by using the Picone's identity (see \cite[Lemma 24]{AG}).

\begin{theorem}\label{Picone}
Let $u \geq 0$ and $v >0$ be differentiable functions, and let $r >1$. Denote
\begin{align*}
\displaystyle
L_r(u,v) & =  \left| \nabla u \right|^r + (r-1) \left( \frac{u}{v} \right)^r \left| \nabla v \right|^r - r \left( \frac{u}{v} \right)^{r-1}\left| \nabla v \right|^{r-2} \nabla v \cdot \nabla u, \\
R_r(u,v) & =  \left| \nabla u \right|^r - \left| \nabla v \right|^{r-2} \nabla v \cdot  \nabla \left(    \frac{u^{r}}{v^{r-1}} \right).
\end{align*}
Then 
$$ \displaystyle L_r(u,v) = R_r(u,v) \geq 0.$$
Moreover, $L_r(u,v)=0$ a.e. in $\Omega$ if and only if $u=kv$ for some constant $k$. 
\end{theorem}

\begin{proposition}\label{funzionicambianosegno}
Let $(\varphi,\psi)$ be an eigenfunction of \eqref{EigenDeThélin0} associated to the eigenvalue $\lambda$ with $\lambda > \lambda_1$. Then both $\varphi$ and $\psi$ change sign. 
\end{proposition}

\begin{proof} 
Let $(\varphi,\psi)$ be an eigenfunction of \eqref{EigenDeThélin0} associated to the eigenvalue $\lambda$ with $\lambda > \lambda_1$. By contradiction, assume that both $\varphi$ and $\psi$ have constant sign, for example $\varphi \geq 0$ and $\psi \geq 0$. We know that 
 $\varphi \in C^{1,\eta}(\overline{\Omega})$ and $\psi \in C^{1,\xi}(\overline{\Omega})$ for some $\eta, \xi \in (0,1)$, and that $\varphi > 0, \psi >0$ in $\Omega$.

Using that  $(\varphi_0, \psi_0)$ is  an eigenfunction associated to $\lambda_1$ and $(\varphi,\psi)$ is an eigenfunction  associated to $\lambda$, for any $(u,v) \in X$ we have
\begin{align}\label{autofunzionephi0psi0}
\displaystyle 
& \into \left| \nabla \varphi_0  \right|^{p-2} \nabla \varphi_0 \cdot  \nabla u \, dx + \into \left| \nabla \psi_0   \right|^{q-2} \nabla \psi_0 \cdot  \nabla v \, dx \\
= & \lambda_1 \left[ \into |\varphi_0|^{\alpha-1} |\psi_0|^{\beta+1} \varphi_0 u \, dx   + \into |\varphi_0|^{\alpha+1} |\psi_0|^{\beta-1} \psi_0 v \, dx     \right], \nonumber
\end{align}
and
\begin{align*}
\displaystyle 
& \into \left| \nabla \varphi  \right|^{p-2} \nabla \varphi \cdot  \nabla u \, dx + \into \left| \nabla \psi   \right|^{q-2} \nabla \psi \cdot  \nabla v \, dx \\
= & \lambda \left[ \into |\varphi|^{\alpha-1} |\psi|^{\beta+1} \varphi u \, dx   + \into |\varphi|^{\alpha+1} |\psi|^{\beta-1} \psi v \, dx     \right]. \nonumber
\end{align*}
By testing \eqref{autofunzionephi0psi0} in $ \displaystyle (u,v)= \left( \frac{\alpha + 1}{p}  \varphi_0  , \frac{\beta +1}{q}  \psi_0\right), $  we get
\begin{align*}
\displaystyle
\frac{\alpha + 1}{p} \into \left| \nabla \varphi_0 \right|^p \, dx  + \frac{\beta + 1}{q} \into \left| \nabla \varphi_0 \right|^p \, dx  = \lambda_1 \into \varphi_0^{\alpha + 1} \psi_0^{\beta +1}    \, dx.
\end{align*}
By testing second equation in
$ (u,v)= \displaystyle \left(  \frac{\alpha + 1}{p} \frac{\varphi_0^{p}}{\varphi^{p-1}} ,\frac{\beta + 1}{q} \frac{\psi_0^{q}}{\psi^{q-1}}    \right),$
we get
\begin{align*}
\displaystyle
& \frac{\alpha + 1}{p} \into \left| \nabla \varphi  \right|^{p-2} \nabla \varphi \cdot  \nabla  \left( \frac{\varphi_0^{p}}{\varphi^{p-1}}    \right)  \, dx  +  \frac{\beta + 1}{q} \into \left| \nabla \psi  \right|^{q-2} \nabla \psi \cdot  \nabla  \left( \frac{\psi_0^{q}}{\psi^{q-1}}    \right)  \, dx \\
= &  \lambda \left[  \frac{\alpha + 1}{p} \into \varphi^{\alpha+1-p} \psi^{\beta+1} \varphi_0^{p} \, dx  +  \frac{\beta + 1}{q} \into \varphi^{\alpha+1} \psi^{\beta+1-q} \psi_0^{q} \, dx     \right].
\end{align*}

\medskip
\noindent
By Theorem \ref{Picone}, and recalling that $\lambda_1 < \lambda,$   we have
\begin{align}\label{dimfunznoncambsegn}
\displaystyle 
0 
& \leq \frac{\alpha+1}{p} \into L_p(\varphi_0, \varphi) \, dx + \frac{\beta+1}{q} \into L_q(\psi_0, \psi) \, dx \\
& = \frac{\alpha+1}{p} \into R_p(\varphi_0, \varphi) \, dx + \frac{\beta+1}{q} \into R_q(\psi_0, \psi) \, dx   \nonumber\\
& = \frac{\alpha+1}{p} \into \left| \nabla \varphi_0  \right|^p \, dx - \frac{\alpha+1}{p} \into \left| \nabla \varphi \right|^{p-2} \nabla \varphi \cdot  \nabla \left(    \frac{\varphi_0^{p}}{\varphi^{p-1}} \right) \, dx \nonumber \\
& \quad + \frac{\beta+1}{q} \into \left| \nabla \psi_0  \right|^q \, dx  \nonumber  - \frac{\beta+1}{q} \into \left| \nabla \psi \right|^{q-2} \nabla \psi \cdot  \nabla \left(    \frac{\psi_0^{q}}{\psi^{q-1}} \right) \, dx \nonumber \\
& =  \lambda_1 \into \varphi_0^{\alpha + 1} \psi_0^{\beta +1}    \, dx         \nonumber \\
& \quad - \lambda \left[  \frac{\alpha + 1}{p} \into \varphi^{\alpha+1-p} \psi^{\beta+1} \varphi_0^{p} \, dx  +  \frac{\beta + 1}{q} \into \varphi^{\alpha+1} \psi^{\beta+1-q} \psi_0^{q} \, dx     \right]\nonumber \\
& \leq \lambda  \left[ \into ( \, \varphi_0^{\alpha + 1} \psi_0^{\beta +1}     -  \frac{\alpha + 1}{p} \varphi^{\alpha+1-p} \psi^{\beta+1} \varphi_0^{p}   - \frac{\beta + 1}{q} \varphi^{\alpha+1} \psi^{\beta+1-q} \psi_0^{q} \, ) \, dx    \right].  \nonumber
\end{align}
However, by Young inequality with conjugated exponents $p/(\alpha + 1)$ and $q/(\beta +1)$, we also have
\begin{align*}
\displaystyle
\varphi_0^{\alpha + 1} \psi_0^{\beta +1}  
& = \left( \varphi_0^{\alpha + 1} \frac{\psi^{\frac{(\alpha+1)(\beta+1)}{p}}}{\varphi^{\frac{(\alpha+1)(\beta+1)}{q}}}  \right) \, \left(  \psi_0^{\beta +1} \frac{\varphi^{\frac{(\alpha+1)(\beta+1)}{q}}}{\psi^{\frac{(\alpha+1)(\beta+1)}{p}}} \right) \\
& \leq \frac{\alpha+1}{p} \varphi_0^p \psi^{\beta+1} \varphi^{\alpha+1-p} + \frac{\beta+1}{q} \psi_0^q \varphi^{\alpha+1}  \psi^{\beta+1-q},
\end{align*}
whence, considering \eqref{dimfunznoncambsegn}, we get 
$$ \displaystyle L_p(\varphi_0, \varphi) = 0 \qquad \text{and} \qquad L_q(\psi_0, \psi)=0.$$
Applying again Theorem \ref{Picone}, we deduce $\varphi_0 = c_1 \varphi$ and $\psi_0= c_2 \psi$ for some positive constant $c_1$ and $c_2$. Observe now that by \eqref{autofunzionephi0psi0} we infer
$$ \displaystyle \into \left| \nabla \varphi_0 \right|^p \, dx = \lambda_1 \into \varphi_0^{\alpha + 1} \psi_0^{\beta +1}   \, dx = \into \left| \nabla \psi_0 \right|^q \, dx.$$
Substituting $\varphi_0 = c_1 \varphi$ and $\psi_0= c_2 \psi$, we get $c_1^p=c_2^q$. Denoting with $\theta:= 
c_1^p$, we have shown that $(\varphi_0, \psi_0) = (\theta^{\frac{1}{p}} \varphi, \theta^{\frac{1}{q}} \psi)$ by which $(\varphi_0, \psi_0)$ is an eigenfunction associated to $\lambda$, hence a contradiction.

The proof is similar in the other sign cases: either $\psi \leq 0 \leq \varphi,$ or $\psi,\varphi \leq 0$ or 
$\varphi \leq 0 \leq \psi$.
\end{proof}

\medskip
\noindent
{\mbox {\it Proof of Theorem~\ref{lambda1isolatoINTRODUZIONE}.~}
By the characterization of $\lambda_1$, for every eigenvalue $\lambda$ of \eqref{EigenDeThélin0} we have $\lambda \geq \lambda_1$. We prove the theorem by contradiction, assuming that $\{\mu_n\}_n$ is a sequence of eigenvalues of \eqref{EigenDeThélin0} such that $\mu_n > \lambda_1$ for any $n \geq 2$ and  $\mu_n \to \lambda_1$ as $n \to + \infty$. For every $n \geq 2$ consider an eigenfunction $z_n=(\varphi_n,\psi_n) \in X$ associated to $\mu_n$ with 
$$ \|\varphi_n\|_{1,p}^{p}+ \|\psi_n\|_{1,q}^q=1.$$
Hence the sequence $\{z_n\}_n$ is bounded and there exists a subsequence, still denoted by $\lbrace{z_n\rbrace}_n,$ that converges to some $z=(\varphi,\psi)$ weakly in $X$ and strongly in $L^p(\Omega)\times L^q(\Omega)$. Since $\varphi_n$ converges to $\varphi$  strongly in $L^p(\Omega)$ and $\psi_n$ converges to $\psi$  strongly in $L^q(\Omega)$, we have $\varphi_n(x) \to \varphi(x)$ and $\psi_n(x) \to \psi(x)$ a.e. in $\Omega$ as $n \to + \infty$. Moreover, passing to a subsequence, we can assume that  there exists $\bar{\varphi} \in L^{p}(\Omega)$ and $\bar{\psi} \in L^{q}(\Omega)$ such that $|\varphi_n(x)|\leq \bar{\varphi}(x)$ and $|\psi_n(x)|\leq \bar{\psi}(x)$ a.e. in $\Omega$ and for all $n \geq 1$.\\
Since $z_n$ is an eigenfunction associated to $\mu_n$, it satisfies
\begin{align}\label{zneigenfunction}
\displaystyle 
& \into \left| \nabla \varphi_n  \right|^{p-2} \nabla \varphi_n \cdot  \nabla u \, dx + \into \left| \nabla \psi_n   \right|^{q-2} \nabla \psi_n \cdot  \nabla v \, dx \\
= & \mu_n \left[ \into |\varphi_n|^{\alpha-1} |\psi_n|^{\beta+1} \varphi_n u \, dx   + \into |\varphi_n|^{\alpha+1} |\psi_n|^{\beta-1} \psi_n v \, dx     \right], \nonumber
\end{align}
for any $(u,v) \in X$.\\
By testing previous equation in $(u,v)=(\varphi_n - \varphi,0)$, we get
\begin{align}\label{compisol1}
\displaystyle 
\into \left| \nabla \varphi_n  \right|^{p-2} \nabla \varphi_n \cdot  \nabla ( \varphi_n - \varphi) \, dx = \mu_n \into |\varphi_n|^{\alpha-1} |\psi_n|^{\beta+1} \varphi_n (\varphi_n - \varphi) \, dx 
\end{align}
By Young inequality with conjugated exponents $p/(\alpha+1)$ and $q/(\beta+1)$, we get
\begin{align*}
\displaystyle 
&\left|  \, |\varphi_n(x)|^{\alpha-1} |\psi_n(x)|^{\beta+1} \varphi_n(x) (\varphi_n(x) - \varphi(x) )  \,  \right|\\
\leq &  \, |\varphi_n(x)|^{\alpha} |\psi_n(x)|^{\beta+1}(  \left| \varphi_n(x) \right| + \left| \varphi(x) \right| ) \\
\leq & 2 \bar{\varphi}(x)^{\alpha+1} \bar{\psi}(x)^{\beta+1} \\
\leq & 2 \left( \frac{\alpha+1}{p}  \bar{\varphi}(x)^p   + \frac{\beta + 1}{q}   \bar{\psi}(x)^q \right) \in L^1(\Omega).
\end{align*} 
Hence, by dominated convergence theorem we have 
$$ \displaystyle \lim_{n \to \infty} \into |\varphi_n|^{\alpha-1} |\psi_n|^{\beta+1} \varphi_n (\varphi_n - \varphi) \, dx =0.$$
Considering also that $\mu_n \to \lambda_1>0$ and passing to limit as $n \to + \infty$ in \eqref{compisol1}, we get 
$$ \displaystyle \lim_{n \to \infty} \into \left| \nabla \varphi_n  \right|^{p-2} \nabla \varphi_n \cdot  \nabla ( \varphi_n - \varphi) \, dx =0.$$
Now, the weak lower semicontinuity and the convexity of  $\| \cdot \|^p_{1,p}$ imply
\begin{align*}
\|\varphi\|_{1,p}^p & \leq \liminf_{n \to \infty} \|\varphi_n\|_{1,p}^p
\leq \limsup_{n \to \infty} \|\varphi_n\|_{1,p}^p = \limsup_{n \to \infty} \into |\nabla \varphi_n|^p \, dx\\
& \leq \limsup_{n \to \infty} \left[ \into |\nabla \varphi|^p \, dx + p \into |\nabla \varphi_n|^{p-2} \nabla \varphi_n \cdot \nabla  \left(\varphi_n - \varphi \right) \, dx \right] = \|\varphi\|_{1,p}^p,
\end{align*}
i.e. $\|\varphi_n\|_{1,p} \to \|\varphi\|_{1,p},$ and by uniform convexity of $W_0^{1,p}(\Omega)$ we deduce $\varphi_n \rightarrow \varphi$ strongly in $\sob$. Analogously, we get that $\psi_n \rightarrow \psi$ strongly in $\sobq$.\\
In particular,
\[\|\varphi\|_{1,p}^p + \|\psi \|_{1,q}^q =
\lim_{n \to \infty} \left( \|\varphi_n\|_{1,p}^p +\|\psi_n\|_{1,q}^q \right)
=1\]
and $z=(\varphi,\psi) \neq 0.$
By testing \eqref{zneigenfunction} in $(u,v) =\left( \frac{\alpha+1}{p} \varphi_n, \frac{\beta + 1}{q} \psi_n \right)$, we get 
$$ \frac{\alpha+1}{p} \into \left| \nabla \varphi_n \right|^p \, dx + \frac{\beta+1}{q} \into \left| \nabla \psi_n \right|^q \, dx = \mu_n \into \left| \varphi_n \right|^{\alpha+1} \left| \psi_n \right|^{\beta+1} \, dx,$$
and passing to limit as $n \to + \infty$ we deduce
$$ \frac{\alpha+1}{p} \into \left| \nabla \varphi \right|^p \, dx + \frac{\beta+1}{q} \into \left| \nabla \psi \right|^q \, dx = \lambda_1 \into \left| \varphi \right|^{\alpha+1} \left| \psi \right|^{\beta+1} \, dx.$$
By last identity since $z=(\varphi,\psi) \neq (0,0),$ we get $z \in \tilde{X}$ and
$$ \displaystyle \lambda_1= \frac{\displaystyle \frac{\alpha+1}{p} \into \left| \nabla \varphi \right|^p \, dx + \frac{\beta+1}{q} \into \left| \nabla \psi \right|^q \, dx}{ \displaystyle \into \left| \varphi \right|^{\alpha+1} \left| \psi \right|^{\beta+1} \, dx  } = \frac{\Phi(z)}{\Psi(z)},$$
i.e. $z=(\varphi,\psi)$ is an eigenfunction associated to $\lambda_1$.\\
By Theorem \ref{simplicityINTRODUZIONE} and \eqref{InsiemeE1riscrittoINTRODUZIONE} we conclude that only one of the following possibilities occurs:
$$ \displaystyle (\varphi,\psi)= (\varphi_1,\psi_1), \; \; (\varphi,\psi)= (-\varphi_1,-\psi_1), \; \; (\varphi,\psi)= (-\varphi_1,\psi_1), \; \; (\varphi,\psi)= (\varphi_1,-\psi_1).$$

\smallskip
\noindent
Let us assume $(\varphi,\psi)= (\varphi_1,\psi_1)$.\\
By Proposition \ref{funzionicambianosegno} we know that both $\varphi_n$ and $\psi_n$ change sign for any $n \geq 2$, hence both $\varphi_n^-:= -\min\{ \varphi_n, 0    \}$ and  $\psi_n^-:= -\min\{ \psi_n, 0    \}$ are different from zero.\\
By testing \eqref{zneigenfunction} in $(u,v)= \left(  \frac{\alpha + 1}{p} \varphi_n^-, \frac{\beta + 1}{q} \psi_n^- \right)$, we get
\begin{align*}
 \displaystyle \frac{\alpha+1}{p} \int_{\Omega} \left| \nabla \varphi_n^- \right|^p \, dx + \frac{\beta+1}{q} \int_{\Omega} \left| \nabla \psi_n^- \right|^q \, dx 
& =  \, \mu_n \int_{\Omega}  \left|\varphi_n^-\right|^{\alpha+1} \left| \psi_n^-\right|^{\beta+1} \, dx,\\
& =  \, \mu_n \int_{\Omega_n^-}  \left|\varphi_n^-\right|^{\alpha+1} \left| \psi_n^-\right|^{\beta+1} \, dx
\end{align*}
where $\Omega_n^- := \{ x \in \Omega \; : \; \varphi_n(x) < 0 \; \text{ or }  \;  \psi_n(x) < 0 \}$.\\
Applying Young and  H\"older inequalities we obtain
\begin{align*}
\displaystyle \int_{\Omega_n^-}  \left|\varphi_n^-\right|^{\alpha+1} \left| \psi_n^-\right|^{\beta+1} \, dx
& \leq   \left[ \frac{\alpha+1}{p} \int_{\Omega_n^-} \left|\varphi_n^-\right|^{p} \, dx   + \frac{\beta+1}{q} \int_{\Omega_n^-} \left| \psi_n^-\right|^{q} \, dx     \right] \\
& \leq   \left[ \frac{\alpha+1}{p} \left| \Omega_n^- \right|^{\frac{p}{N}} \lVert \varphi_n^-  \rVert_{p^*}^p  +  \frac{\beta+1}{q} \left| \Omega_n^- \right|^{\frac{q}{N}} \lVert \psi_n^-  \rVert_{q^*}^q \right].
\end{align*}
For $p^*=Np/(N-p)$ and $q^*=Nq/(N-q)$  we use now the Sobolev embeddings $\lVert u \rVert_{p^*} \leq \mathcal{S}(p,N) \lVert \nabla u \rVert_{p}$ and $\lVert v \rVert_{q^*} \leq \mathcal{S}(q,N) \lVert \nabla v \rVert_{q}$ to get from the above identity that 
\begin{align*}
\displaystyle
&\frac{\alpha+1}{p}  \int_{\Omega} \left| \nabla \varphi_n^- \right|^p \, dx + \frac{\beta+1}{q} \int_{\Omega} \left| \nabla \psi_n^- \right|^q \, dx  \\
 & \leq 
 \mu_n C \max \left\lbrace \left| \Omega_n^- \right|^{\frac{p}{N}}, \left| \Omega_n^- \right|^{\frac{q}{N}}   \right\rbrace \left[ \frac{\alpha+1}{p} \int_{\Omega} \left| \nabla \varphi_n^- \right|^p \, dx + \frac{\beta+1}{q} \int_{\Omega} \left| \nabla \psi_n^- \right|^q \, dx   \right],
\end{align*}
where $ \displaystyle C:= \max \left\lbrace \mathcal{S}(p,N)^p, \mathcal{S}(q,N)^q \right\rbrace.$ This implies that 
$$ \displaystyle \frac{1}{\mu_n C} \leq \max \left\lbrace \left| \Omega_n^- \right|^{\frac{p}{N}}, \left| \Omega_n^- \right|^{\frac{q}{N}}   \right\rbrace ,$$
whence
\begin{equation*}
 \displaystyle 0 < \frac{1}{C \lambda_1} \leq \liminf_{n \to \infty}  \, \max \left\lbrace \left| \Omega_n^- \right|^{\frac{p}{N}}, \left| \Omega_n^- \right|^{\frac{q}{N}}   \right\rbrace.
\end{equation*}
Since, up to subsequence,  $\varphi_n(x) \to \varphi_1(x)$, $\nabla \varphi_n(x) \to \nabla \varphi_1(x)$, $\psi_n(x) \to \psi_1(x)$, and $\nabla \psi_n(x) \to \nabla \psi_1(x)$ a.e. in $\Omega$, we can apply Severini-Egorov theorem to say that in the complement of a set of arbitrarily small measure we have $\varphi_n$ converges uniformly to $\varphi_1$, $ \nabla \varphi_n$ converges uniformly to $ \nabla \varphi_1$, $\psi_n$ converges uniformly to $\psi_1$ and $\nabla \psi_n$ converges uniformly to $\nabla \psi_1$. Since  $\varphi_1>,$  $\psi_1>0$ in $\Omega$ and $\frac{\partial \varphi_1}{\partial \nu},\frac{\partial \varphi_1}{\partial \nu}<0$ on $\partial \Omega$, we conclude that there exists a subset of $\Omega$ of arbitrarily small measure such that in its complement $\varphi_n$ and $\psi_n$ are positive  for $n$ sufficiently large, that is a contradiction with the last inequality.

Let us recall that in the previous argument we have assumed $(\varphi,\psi)= (\varphi_1,\psi_1)$.\\
In the rest of cases we can argue similarly. For instance, if for example $(\varphi,\psi)= (\varphi_1,-\psi_1)$, then we arrive to a contradiction by considering 
$\varphi_n^-:= -\min\{ \varphi_n, 0    \}$,  $\psi_n^+:= \max\{ \psi_n, 0    \}$, and 
\begin{align*}
\displaystyle
\Omega_n^{-,+} := \{ x \in \Omega \; : \; \varphi_n(x) < 0 \; \text{ or }  \;  \psi_n(x) > 0  \}.
\end{align*}
 
\qed

\section{A SEQUENCE OF EIGENVALUES}

\bigskip
\noindent
Let us introduce the $C^1$ functional $ Q :  X \setminus \{(0,0)\} \to \mathbb{R}$ defined by 
\begin{equation}\label{RQ}
 \displaystyle Q(z):=\frac{ \Psi (z)}{ \Phi(z) },
\end{equation}
where $\Phi$ and $\Psi$ are defined in \eqref{funzionalePhi} and \eqref{funzionalePsi} respectively.\\
Observe that
\begin{equation}\label{derivataQ}
 \displaystyle Q'(z) = \frac{1}{\Phi(z)} \left[ \, \Psi' (z) - Q(z) \, \Phi' (z)    \,  \right] \qquad \forall z=(u,v) \in X \setminus \{(0,0) \}.
\end{equation}
We note that a real number $\lambda \neq 0$ is an eigenvalue of \eqref{EigenDeThélin0} if and only if there exists $\tilde{z} \in \tilde{X}$ such that
$$ \displaystyle \Phi'(\tilde{z})= \lambda \Psi'(\tilde{z}),$$
or equivalently, if $1/\lambda$ is a critical value for $Q$.\\
Let us denote with $\mathcal{M}$  the set
\begin{align}\label{nuovomanifold}
\displaystyle 
\mathcal{M}:=\Biggl\{ z=(u,v) \in X \, : \,  \Phi(z)=1   \Biggr\}.
\end{align}
$\mathcal{M}$  is a $C^1$ manifold of codimension $1$ in $X$ (see, for instance \cite[Example 27.2]{DEM}) and for any $z \in \mathcal{M}$ the tangent space $T_z\mathcal{M}$ of $\mathcal{M}$ at $z$ can be identified with
$$ \displaystyle \text{Ker }( \Phi'(z) )= \{ h \in X \, : \, \langle \Phi'(z),h \, \rangle=0    \}.$$
Observe that by definition of Q and \eqref{derivataQ} the restriction $Q_{\mathcal{M}}:= Q\big|_{ \mathcal{M}}$ of $Q$ to $\mathcal{M}$ satisfies 
$$ \displaystyle   Q_{ \mathcal{M}}(z)=\Psi(z) \quad \text{ and } \quad  Q'(z) =  \Psi' (z) - \Psi(z) \, \Phi' (z) \qquad \forall z=(u,v) \in \mathcal{M}.$$
Moreover, (see \cite[Example 27.3]{DEM}), the derivative $Q_{\mathcal{M}}'$ of $Q_{\mathcal{M}}$ is defined on the tangent bundle $T\mathcal{M}$, and in particular for any $z \in \mathcal{M}$ the map $Q_{\mathcal{M}}'(z)$, defined on $T_z \mathcal{M}$, can be identified  as
\begin{equation}\label{derivataQsuM}
\displaystyle Q_{ \mathcal{M}}'(z)= [ Q' (z) -  \langle \, Q'(z) , e(z)   \, \rangle \, \Phi'(z)]\big|_{T_z \mathcal{M}}, \qquad \forall \,    z=(u,v) \in \mathcal{M},    
\end{equation}
where $e: \mathcal{M} \to X$ is given by 
$$ \displaystyle e(z)= \left( \frac{u}{p} , \frac{v}{q} \right) \qquad \forall \, z=(u,v) \in \mathcal{M}.$$
Notice that
$$ \displaystyle \langle \, \Psi'(z) , e(z)   \, \rangle = \Psi(z) \qquad \text{and} \qquad \langle \, \Phi'(z) , e(z)   \, \rangle = \Phi(z)=1 \qquad \forall z=(u,v) \in \mathcal{M},$$
hence by \eqref{derivataQsuM} we deduce
\begin{equation}\label{derivataQsuM3}
\displaystyle Q_{ \mathcal{M}}'(z)=  [\Psi' (z) -  \Psi(z) \, \Phi'(z)]\big|_{T_z \mathcal{M}} =Q'(z)\big|_{T_z \mathcal{M}} \qquad \forall z=(u,v) \in \mathcal{M}.
\end{equation}

We recall that a real number $c$ is a critical value of $Q_{ \mathcal{M}}$ if there exists $z \in \mathcal{M}$ such that $Q_{ \mathcal{M}}(z)=c$ and $ Q_{ \mathcal{M}}'(z) \equiv 0$. It follows from standard arguments that critical values of $Q_{ \mathcal{M}}$ correspond to critical values of $Q$. In particular,  we deduce that $\lambda \neq 0$ is an eigenvalue of \eqref{EigenDeThélin0} if and only if $1/\lambda$ is a critical value for $Q_{ \mathcal{M}}$.

Given $z \in \mathcal{M}$, we have $Q_{ \mathcal{M}}'(z) \in T^*_z \mathcal{M},$ where $ T^*_z \mathcal{M}$ denotes the dual space of $ T_z \mathcal{M}$ endowed with the norm $ \lVert \cdot \rVert_z^*$. Denoting with $\lVert \cdot \rVert^*$ the norm of $X^*$, by \cite[Proposition 3.54]{PAO} it follows that
\begin{align}\label{normaspaziotangente}
\displaystyle   \lVert Q_{ \mathcal{M}}'(z) \rVert^*_z= \min_{\mu \in \mathbb{R}} \lVert \,  Q'(z)  - \mu \Phi'(z) \, \rVert^*=\min_{\mu \in \mathbb{R}} \lVert \,  \Psi'(z)  -( \Psi(z) +  \mu) \Phi'(z) \, \rVert^* .
\end{align}

As it has been mentioned in the introduction,  the construction of  a sequence of eigenvalues of \eqref{EigenDeThélin0} uses a deformation lemma for $C^1$ submanifolds of a Banach space (see Theorem~\ref{BON}) based in \cite[Theorem 2.5]{BON} that requires a coupled Palais-Smale condition.

Firstly we observe that $0$ is the only critical value of $\Phi$, therefore for any $\varepsilon >0$ sufficiently small, the manifold
$$ \mathcal{M}_{\varepsilon}:= \Biggl\{ z =(u,v) \in X \, : \, \Phi(z)=1+ \varepsilon \Biggr\} $$
is well defined, hence definitions and results proved for $Q$ on $\mathcal{M}$ still hold on $\mathcal{M}_{\varepsilon}$.

\begin{lemma}\label{QPhiPSaccoppiata}
For every $c>0$ the functionals $ \displaystyle Q$ and $\Phi$ satisfy the following coupled Palais-Smale condition at level $c$ on $\mathcal{M}$ (denoted by $\widehat{(PS)}_{\mathcal{M},c}$), that is  any sequence $\{z_n\}_n = \{(u_n,v_n)\}_n \subset X$ such that
\begin{itemize}
\item[$i)$] $z_n \in \mathcal{M}_{\varepsilon_n},$ where $\varepsilon_n >0$ and $\varepsilon_n \to 0$ as $n \to \infty$;
\item[$ii)$] $Q(z_n) \to c$ as $n \to \infty$;
\item[$iii)$] $\lVert Q_{\mathcal{M}_{\varepsilon_n}}'(z_n) \rVert_{z_n}^* \to 0 \text{ or } \lVert \Phi'(z_n) \rVert \to 0$ as $n \to \infty$;
\end{itemize}
has a converging subsequence in $X$.
\end{lemma}

\begin{proof} Let $\{z_n\}_n = \{(u_n,v_n)\}_n \subset X$ a sequence satisfying $i)-iii)$.\\
By condition $i)$ we deduce that 
\begin{align*}
\displaystyle \lVert u_n \rVert_{1,p}^p \leq \frac{p}{\alpha+1} \Phi(z_n) = \frac{p(1 + \varepsilon_n)}{\alpha+1} \quad \text{and} \quad \lVert v_n \rVert_{1,q}^q \leq \frac{q}{\beta+1} \Phi(z_n) = \frac{q(1 + \varepsilon_n)}{\beta+1},
\end{align*}
that is $ \lbrace{ z_n \rbrace}_n $ is bounded in $X$, and so there exists a subsequence, still denoted by $\lbrace{z_n\rbrace}_n,$ that converges to some $z=(u,v)$ weakly in $X$ and strongly in $L^p(\Omega)\times L^q(\Omega)$.
By condition $iii)$ we have two possibilities.\\
Firstly we assume that $\lVert \Phi'(z_n) \rVert \to 0$ as $n \to \infty$ holds true. In particular, we have
$$ \langle \Phi'(z_n) , u_n - u \rangle \to 0,$$
that is
$$ \displaystyle \into \left| \nabla u_n \right|^{p-2} \nabla u_n \cdot \nabla (u_n - u) \, dx \to 0.$$
Now, the weak lower semicontinuity and the convexity of  $\| \cdot \|^p_{1,p}$
imply
\begin{align*}
\|u\|_{1,p}^p & \leq \liminf_{n \to \infty} \|u_n\|_{1,p}^p
\leq \limsup_{n \to \infty} \|u_n\|_{1,p}^p = \limsup_{n \to \infty} \into |\nabla u_n|^p \, dx\\
& \leq \limsup_{n \to \infty} \left[ \into |\nabla u|^p \, dx + p \into |\nabla u_n|^{p-2} \nabla u_n \cdot \nabla  \left(u_n-u\right) \, dx \right] = \|u\|_{1,p}^p,
\end{align*}
i.e. $\|u_n\|_{1,p}^p \to \|u\|_{1,p}^p$, and  by uniform convexity of $W_0^{1,p}(\Omega)$, we deduce $u_n\rightarrow u$ in $\sob$. Analogously,  we get that $v_n\rightarrow v$ strongly in $\sobq$.

Now, let us assume the other possibility, i.e.  $\lVert Q_{\mathcal{M}_{\varepsilon_n}}'(z_n) \rVert_{z_n}^* \to 0 $ as $n \to \infty$.
Strong convergences of $u_n$ to $u$ in $L^{p}(\Omega)$ and of $v_n$ to $ v$ in $L^{q}(\Omega)$ give $ \displaystyle \Psi(z_n) \to \Psi(z).$
Moreover, by $i)$ we infer that $\Phi(z_n) \to 1$. By using also $ii)$, we get
$$ \displaystyle \Psi(z_n) \to \Psi(z)=c.$$
By \eqref{normaspaziotangente} applied to $z_n \in \mathcal{M}_{\varepsilon_n}$ and the convergence to zero of $\lVert Q_{\mathcal{M}_{\varepsilon_n}}'(z_n) \rVert_{z_n}^*$, there exists a sequence $\{\mu_n\}_n \subset \mathbb{R}$ such that
$$ \displaystyle \Psi'(z_n)  -( \Psi(z_n) +  \mu_n) \Phi'(z_n) \to 0 \quad \text{in } X^* \qquad \text{as } n \to \infty.$$
Testing on the vector $(u_n/p,v_n/q)$, we get 
$$ \displaystyle \Psi(z_n)  -( \Psi(z_n) +  \mu_n) \Phi(z_n) \to 0 \qquad \text{as } n \to \infty,$$
and using that $\Phi(z_n) \to 1$ and $\Psi(z_n) \to c$, we deduce $\mu_n \to 0$ as $n \to \infty$.
Testing now on the vector $(u_n - u,0),$ we get
$$ \displaystyle \into |u_n|^{\alpha-1}  |v_n|^{\beta +1 } u_n (u_n -u) \, dx - ( \Psi(z_n) +  \mu_n)\into \left| \nabla u_n \right|^{p-2} \nabla u_n \cdot \nabla (u_n - u) \, dx \to 0,  $$
and considering the strong convergences of $u_n$ to $u$ in $L^p(\Omega)$ and $v_n$ to $v$ in $L^q(\Omega)$ we deduce
$$ \displaystyle ( \Psi(z_n) +  \mu_n)\into \left| \nabla u_n \right|^{p-2} \nabla u_n \cdot \nabla (u_n - u) \, dx \to 0.   $$
Since $\Psi(z_n) +  \mu_n\to c$ with $c \neq 0$, we get
\begin{equation*}
\displaystyle 
\into \left| \nabla u_n \right|^{p-2} \nabla u_n \cdot \nabla (u_n - u) \, dx \to 0. 
\end{equation*}
Arguing as before we conclude that $u_n \to u$ strongly in $W_0^{1,p}(\Omega)$. Similarly we deduce that $v_n \to v$ strongly in $W_0^{1,q}(\Omega)$.
\end{proof}
\noindent
We will apply the following deformation result for $C^1$ submanifolds of a Banach space, which is a consequence of \cite[Theorem 2.5]{BON}.

\begin{theorem}\label{BON}
Let $X$ be a Banach space and $\Phi \in C^1(X, \mathbb{R})$. Assume that $1$ is a regular value of $\Phi$. If $Q$ is a $C^1$ functional on a neighborhood of $\mathcal{M}= \{ x \in X \, : \, \Phi(x)=1 \}$ and $Q$ and $\Phi$ satisfy $\widehat{(PS)}_{\mathcal{M},c}$ for  a regular value $c \in \mathbb{R}$ of $Q_\mathcal{M}$, then
there exists $\hat{\varepsilon} >0$ such that for all $0 < \varepsilon < \hat{\varepsilon}$ there is an homeomorphism $\eta$ of $\mathcal{M}$ onto $\mathcal{M}$ such that:
\begin{enumerate}
\item $ \eta (z)=z$  if $Q(z) \not \in [ c - \hat{\varepsilon}, c + \hat{\varepsilon}]$;
\item $Q(\eta(z)) \geq Q(z)$ for all $z \in \mathcal{M}$;
\item $Q(\eta(z)) \geq c + \varepsilon$ for all $z \in \mathcal{M}$ such that $Q(z) \geq c - \varepsilon$;
\item If $\mathcal{M}$ is symmetric $(\mathcal{M}=-\mathcal{M})$ and if $Q$ is even, then $\eta$ is odd.
\end{enumerate}
\end{theorem}
Now, we will construct a sequence of eigenvalues for \eqref{EigenDeThélin0}.
\begin{proposition}\label{prop 2.9}
Let $k \in \mathbb{N}$ and let $S^{k-1}$ be the unit sphere in $\mathbb{R}^k$. Set 
$$
\mathcal{M}_k := \{ A=\alpha(S^{k-1}) \subset \mathcal{M}  \ | \,  \alpha: S^{k-1} \to \mathcal{M} \text{ is odd and continuous}  \}.
$$
If
\begin{equation}\label{ck}
	c_k := \sup_{A \in \mathcal{M}_k} \min_{(u,v) \in A} Q_{ \mathcal{M}}(u,v),
\end{equation}
then
\begin{equation}\label{lambdakSfera}
	\lambda_k := \frac{1}{c_k}
\end{equation}
is an eigenvalue of problem \eqref{EigenDeThélin0}.
\end{proposition}

\begin{proof}
As it has been mentioned a real number $\lambda$ is an eigenvalue of \eqref{EigenDeThélin0} if and only if $1/\lambda$ is a critical value of $ Q_{ \mathcal{M}}$. By contradiction assume  that $c_k$  is a regular value of $ Q_{ \mathcal{M}}$. By Theorem \ref{BON} there exists $ \varepsilon >0$ an homeomorphism $\eta$ of $\mathcal{M}$ onto $\mathcal{M}$ satisfying  properties $(1)-(4)$. By definition of $c_k$, we deduce that there exists an odd and continuous map $\alpha: S^{k-1} \to \mathcal{M} $ such that $ \min_{(u,v) \in A} Q_{ \mathcal{M}}(u,v) > c_k - \varepsilon$ for $A= \alpha(S^{k-1})$. Hence, taking into account property $(4)$, we deduce that $\eta(A) \in \mathcal{M}_k$. However, by $(3)$ of the cited Theorem we infer that $ \min_{(u,v) \in \eta(A) } Q_{ \mathcal{M}}(u,v) \geq c_k + \varepsilon$, which contradicts the definition of $c_k$. Therefore $c_k$ is a critical value of $ Q_{ \mathcal{M}}$ and $\lambda_k$ is an eigenvalue of \eqref{EigenDeThélin0}.

\end{proof}

\begin{remark}\label{definizionicoincidono}
We point out that the first eigenvalue $\lambda_1$ defined in \eqref{lambdakSfera} coincides with the eigenvalue $\lambda_1$ defined in \eqref{lambda1isolatoINTRODUZIONE}. In fact, by definition we know that 
$$
\mathcal{M}_1 := \{ A=\{ z,-z\}   \ | \, z\in \mathcal{M}\},
$$
and since $\Psi (-z)=\Psi (z)$ we also have
$$ 
\displaystyle c_1 
= \sup_{A \in \mathcal{M}_1} \min_{ \bar{z} \in A} \Psi(\bar{z})
= \sup_{ z \in \mathcal{M}}  \Psi(z) 
=  \sup_{z \in X \setminus (0,0)} \frac{ \Psi (z)}{ \Phi(z) } 
= \frac{1}{ \displaystyle \inf_{z \in \tilde{X}} \frac{\Phi(z)}{\Psi(z)}}= \frac{1}{\lambda_1}.
$$
\end{remark}

\section{PALAIS-SMALE CONDITION OF $J$}

\medskip
\noindent
In order to study problem \eqref{Syst0}, let us define the functional $J:X \to \mathbb{R}$ as
\begin{align}\label{J}
\displaystyle
J(z)
&:=  \frac{\alpha+1}{p} \into \left| \nabla u \right|^p \, dx + \frac{\beta+1}{q} \into \left| \nabla v \right|^q \, dx    - \lambda_1   \into \left| u \right|^{\alpha+1} \left| v \right|^{\beta+1} \, dx \\
& \quad - \into F(x,u,v) \, dx + \into h_1 u \, dx + \into h_2 v \, dx \nonumber \\
& = \Phi(u,v) - \lambda_1 \Psi(u,v) - \into F(x,u,v) \, dx + \into h_1 u \, dx + \into h_2 v \, dx \nonumber
\end{align}
for any $z=(u,v) \in X$.\\
It turns out that the functional $J$ is of class $C^1$ on $X$ with 
\begin{align*}
& \langle J'(z_0), z \rangle  \\
= & \displaystyle (\alpha+1) \into |\nabla
		u_0|^{p-2}\nabla u_0 \cdot \nabla u \ dx + (\beta+1)\into |\nabla v_0|^{q-2}\nabla v_0 \cdot \nabla v \ dx  \\
 & - \displaystyle \lambda_1 \left[ (\alpha+1)\into | u_0|^{\alpha-1}| v_0|^{\beta+1} u_0 u \ dx + (\beta+1)\into | u_0|^{\alpha+1}| v_0|^{\beta-1} v_0  v \ dx  \right] \\
 & -\displaystyle \into \left[  F_s(x,u_0,v_0) u +  F_t(x,u_0,v_0) v \right] \, dx +  \into h_1 u \, dx + \into h_2 v \, dx. 
\end{align*}
for any $z_0=(u_0,v_0), z=(u, v) \in X$.
\begin{lemma}
Let $F:\Omega \times \mathbb{R}^2 \to \mathbb{R}$ be a $C^1$-Carathéodory function satisfying \textbf{(F1)} and \textbf{(F2)}. Consider $h_1 \in L^{p'}(\Omega)$ and  $h_2 \in L^{q'}(\Omega)$. If we assume that 
\begin{itemize}
\item[•]  in the case $p<q$,  either condition \eqref{LLpminoreq1} or  condition \eqref{LLpminoreq2} holds true;

\item[•]  in the case $p=q$,  either condition \eqref{LLpp1} or  condition \eqref{LLpp2} holds true;

\item[•]  in the case $p>q$,  either condition \eqref{LLpmaggioreq1} or  condition \eqref{LLpmaggioreq2} holds true;
\end{itemize}
then $J$ satisfies the (PS) condition, namely if $ \lbrace{ z_n \rbrace}_n = \lbrace{ (u_n,v_n) \rbrace}_n$ is a sequence in $X$ such that there exists a positive constant $c$ such that  
\begin{equation}\label{PS1}
\left| J(z_n)\right| \leq c  \quad \text{ for any } n \in \mathbb{N}
\end{equation}
and 
\begin{equation}\label{PS2}
\|J'(z_n)\|~\rightarrow~0,
\end{equation}
then $\lbrace{ z_n \rbrace}_n $ has a convergent subsequence in $X$.
\end{lemma}

\begin{proof}
Let $\lbrace{ z_n \rbrace}_n=\lbrace{ (u_n,v_n) \rbrace}_n \subset X$ satisfying \eqref{PS1} and \eqref{PS2}. Let us begin by proving that the sequence $\lbrace{ z_n \rbrace}_n=\lbrace{ (u_n,v_n) \rbrace}_n$  is bounded in $X$. By contradiction, there exists a subsequence, still denoted with $\lbrace{ z_n \rbrace}_n,$  such that $\|(u_n,v_n)\|~\rightarrow~\infty$.
In particular, as $n \to + \infty$ we have
$$r_n:=\|u_n\|_{1,p}^{p}+ \|v_n\|_{1,q}^q \to + \infty.$$
Let us set
$$ \displaystyle \bu_n= \frac{u_n}{r_n^{\frac{1}{p}}} \qquad \text{and} \qquad  \bv_n=\frac{v_n}{r_n^{\frac{1}{q}}},$$
and observe that $\bz_n:=(\bu_n,\bv_n)$ satisfies 
$$ \displaystyle \lVert \bu_n \rVert_{1,p}^p + \lVert \bv_n \rVert_{1,q}^q = \frac{\|u_n\|_{1,p}^{p}+ \|v_n\|_{1,q}^q }{r_n}= 1 \qquad \text{ for any } n \in \mathbb{N}.$$
Hence $ \lVert \bz_n \rVert \leq 2$ for any $n \in \mathbb{N}$ and there exists a subsequence, still denoted by $\lbrace{\bz_n\rbrace}_n,$ that converges to some $\bz=(\bu,\bv)$ weakly in $X$ and strongly in $L^p(\Omega)\times L^q(\Omega)$.\\
In particular, as $n \to \infty$ we have
\begin{align}\label{ps1}
\frac{1}{\ r_n^{\frac{p-1}{p}}}\ \bigl\langle J'(z_n),(\bu_n-\bu,0)\bigr\rangle
& = \displaystyle  (\alpha+1)\into |\nabla \bu_n|^{p-2}\nabla \bu_n \cdot \nabla \left( \bu_n - \bu \right) \ dx \\
& \displaystyle \quad -\lambda_1 (\alpha+1)\into | \bu_n|^{\alpha-1} | \bv_n|^{\beta+1}\bu_n \left( \bu_n - \bu \right) \ dx \nonumber \\
& \displaystyle \quad - \frac{1}{\ r_n^{\frac{p-1}{p}}} \into  F_s(x,u_n,v_n) (\bu_n - \bu)   \, dx \nonumber \\
& \displaystyle \quad + \frac{1}{\ r_n^{\frac{p-1}{p}}} \into h_1(\bu_n - \bu)  \, dx \to 0, \nonumber
\end{align}
and
\begin{equation}\label{palaissmale2}
\frac{1}{\ r_n^{\frac{q-1}{q}}}\ \bigl\langle J'(z_n),(0,\bv_n-\bv)\bigr\rangle \rightarrow 0.
\end{equation}
Since $\bu_n \to \bu$ in $L^p(\Omega)$, $h \in L^{p'}(\Omega)$ and $F_s$ is bounded  we deduce as $n \to + \infty$ that
\begin{equation*}
\displaystyle  \frac{1}{\ r_n^{\frac{p-1}{p}}} \into h_1(\bu_n - \bu)  \, dx \, \to 0 \quad \text{ and } \quad  \frac{1}{\ r_n^{\frac{p-1}{p}}} \into  F_s(x,u_n,v_n) (\bu_n - \bu)   \, dx \, \to 0.
\end{equation*}
Moreover, considering $(\alpha +1)/p + (\beta + 1)/q=1$ and the strong convergences $\bu_n \to \bu$ in $L^{p}(\Omega)$ and $\bv_n \to \bv$ in $L^{q}(\Omega)$, by dominated convergence theorem we get
$$ \displaystyle \lim_{n \to \infty} \into | \bu_n|^{\alpha-1} | \bv_n|^{\beta+1}\bu_n \left( \bu_n - \bu \right) \, dx = 0.$$
Therefore, taking into account these convergences in \eqref{ps1}, we have proved
$$ \displaystyle \lim_{n \to \infty} \into |\nabla \bu_n|^{p-2}\nabla \bu_n \cdot \nabla \left( \bu_n - \bu \right) \, dx  = 0.$$
How done in the proof of Theorem \ref{lambda1isolatoINTRODUZIONE}, by using also \eqref{palaissmale2} we finally obtain that $\{\bu_n,\bv_n\}_n$ converges to $\bz=(\bu,\bv)$ strongly in $X$, where $\bz \neq (0,0)$ satisfies
$$ \displaystyle \lVert \bu \rVert_{1,p}^p + \lVert \bv \rVert_{1,q}^q=1.$$
Now, to prove that $\bz$ is an eigenfunction associated to the eigenvalue $\lambda_1$, it suffices to show that $\bu \neq 0$, $\bv \neq 0$ and 
$$ \displaystyle \frac{\Phi(\bz)}{\Psi(\bz)}=\lambda_1.$$
In the following, we  assume both $\lVert u_n \rVert_{1,p} \to \infty$ and $\lVert v_n \rVert_{1,q} \to \infty$. In the case that one of both sequences is bounded,  computations are simpler and left to reader.
Considering just the boundedness form above coming from \eqref{PS1}, we have 
$$\displaystyle \Phi(u_n,v_n) - \lambda_1 \Psi(u_n,v_n) - \into F(x,u_n,v_n) \, dx + \into h_1 u_n \, dx + \into h_2 v_n \, dx \leq c.$$
Dividing by $r_n$ and taking into account the definition of $\bu_n$ and $\bv_n$, we have
$$\displaystyle \Phi(\bu_n,\bv_n) - \lambda_1 \Psi(\bu_n,\bv_n) -  \into \frac{F(x,u_n,v_n)}{r_n} \, dx + \into \frac{h_1 u_n}{r_n} \, dx +  \into \frac{h_2 v_n}{r_n} \, dx \leq \frac{c}{r_n},$$
whence
\begin{align}\label{limsup}
\displaystyle \limsup_{n \to  \infty} \Biggl[ 
& \Phi(\bu_n,\bv_n) - \lambda_1 \Psi(\bu_n,\bv_n)  \\
&  -  \into \frac{F(x,u_n,v_n)}{r_n} \, dx + \into \frac{h_1 u_n}{r_n} \, dx +  \into \frac{h_2 v_n}{r_n} \, dx \Biggr] \leq 0. \nonumber
\end{align}
Since $h_1 \in L^{p'}(\Omega)$ and $h_2 \in L^{q'}(\Omega)$, by applying H\"older and Poincaré inequality and the convergence to infinity of the sequence $\{r_n\}_n$, we have 
$$ \displaystyle \lim_{n \to \infty} \into  \frac{h_1 u_n}{r_n} \, dx = 0 \qquad \text{and}  \qquad \lim_{n \to \infty} \into \frac{h_2 v_n}{r_n} \, dx =0.$$
For almost every $(x,s,t) \in \Omega \times \mathbb{R}^2$ mean value theorem gives
$$ \displaystyle F(x,s,t)=  F(x,0,0) + \int_0^1 F_s(x,\tau s,\tau t)s \, d \tau + \int_0^1 F_t(x,\tau s,\tau t)t \, d \tau,$$
hence by assumption \textbf{(F1)} we deduce
$$ \displaystyle \left| F(x,s,t)   \right| \leq \left| F(x,0,0) \right| + M |s| + M|t|.$$
Therefore,
$$ \displaystyle \left| \into \frac{F(x,u_n,v_n)}{r_n} \, dx \right| \leq \frac{\lVert F(x,0,0) \rVert_1}{r_n} + K_1 \lVert u_n \rVert_{1,p}^{1-p} + K_2 \lVert v_n \rVert_{1,q}^{1-q},$$
whence
$$ \displaystyle \lim_{n \to \infty} \into \frac{F(x,u_n,v_n)}{r_n} \, dx =  0.$$
Considering these convergences to zero into \eqref{limsup}, we obtain
$$\displaystyle \limsup_{n \to  \infty} \Biggl[ \Phi(\bu_n,\bv_n) - \lambda_1 \Psi(\bu_n,\bv_n) \Biggr] \leq 0.$$
Now, since $\bu_n \, \to \bu$ in $L^p(\Omega)$ and $\bv_n \, \to \bv$ in $L^q(\Omega)$, by Lebesgue theorem we deduce
$$ \displaystyle \Psi(\bu_n,\bv_n) \, \to \Psi(\bu,\bv) \qquad   \text{ as } n \, \to \infty.$$
and by the strong convergences $\bu_n \, \to\bu$ in $\sob$ and $\bv_n \, \to \bv$ in $\sobq$ we also get
$$ \displaystyle \Phi(\bu_n,\bv_n) \, \to \Phi(\bu,\bv) \qquad   \text{ as } n \, \to \infty.$$
Therefore, we have proved $\Phi(\bu,\bv) \leq \lambda_1 \Psi(\bu,\bv)$, i.e. 
$$ \displaystyle  \frac{\alpha+1}{p} \into \left| \nabla \bu \right|^p \, dx + \frac{\beta+1}{q} \into \left| \nabla \bv \right|^q \, dx    \leq \lambda_1   \into \left| \bu \right|^{\alpha+1} \left| \bv \right|^{\beta+1} \, dx. $$
Since $\bz \neq (0,0)$, the left side of previous inequality is strictly greater than zero, by which $\bu \neq 0,$ $\bv \neq 0$ and
$$ \displaystyle \frac{\Phi(\bz)}{\Psi(\bz)}=\lambda_1,$$
that is $\bz=(\bu,\bv)$ is an eigenfunction associated to the eigenvalue $\lambda_1$.\\
Since also $\bz$ satisfies $ \displaystyle \lVert \bu \rVert_{1,p}^p + \lVert \bv \rVert_{1,q}^q=1,$ we conclude by Theorem \ref{simplicityINTRODUZIONE} and \eqref{InsiemeE1riscrittoINTRODUZIONE} that only one of the following possibilities occurs:
$$ \displaystyle (\bu,\bv)= (\varphi_1,\psi_1), \; \; (\bu,\bv)= (-\varphi_1,-\psi_1), \; \; (\bu,\bv)= (-\varphi_1,\psi_1), \; \; (\bu,\bv)= (\varphi_1,-\psi_1).$$
By \eqref{PS2}, we deduce that, up to subsequence, there exists a decreasing sequence $\{ \varepsilon_n \}_n$, converging to zero, such that
$$ \displaystyle \left| \left\langle J'(z_n), \left(\frac{u_n}{p},\frac{v_n}{q}\right) \right\rangle \right| \leq \varepsilon_n \lVert z_n \rVert \qquad \forall n \in \mathbb{N}, \quad \forall z=(u,v) \in X.$$
Hence, taking into account also \eqref{PS1},  we have
\begin{equation*}
\displaystyle  \left(1 +  \varepsilon_n \lVert z_n \rVert \right) C \geq \left| J(z_n) - \left\langle J'(z_n), \left(\frac{u_n}{p},\frac{v_n}{q}\right)\right\rangle + \into F(x,0,0) \, dx \right|,
\end{equation*}
which by the mean value theorem and the definition of $\bu_n$ and $\bv_n$ implies
\begin{align}\label{calcolo}
\displaystyle
\left(1 +  \varepsilon_n \lVert z_n \rVert \right) C \geq
& \left| r_n^{\frac{1}{p}} \Biggl\{ \left( 1 - \frac{1}{p}  \right) \into h_1 \bu_n \, dx  \ -  \into  \int_0^1 F_s(x,\tau u_n,\tau v_n) \bu_n \, d \tau \, dx  \nonumber \right.\\
+&  \frac{1}{p} \into F_s(x,u_n,v_n) \bu_n \,  dx \Biggr\} \\ 
+& r_n^{\frac{1}{q}} \Biggl\{ \left( 1 - \frac{1}{q}  \right) \into h_2 \bv_n \, dx  \ -  \into  \int_0^1 F_t(x,\tau u_n,\tau v_n)\bv_n\, d \tau \, dx  \nonumber\\
+&   \left. \frac{1}{q} \into F_t(x,u_n,v_n)\bv_n \,  dx \Biggr\}\right| .  \nonumber
\end{align}
Let us assume $(\bu,\bv)= (\varphi_1,\psi_1)$.\\
Since $ \bu_n \to \varphi_1$ in $L^p(\Omega)$ and  $ \bv_n \to \psi_1$ in $L^q(\Omega)$, we have
$$\displaystyle \lim_{n \to \infty} \into h_1 \bu_n \, dx  = \into h_1 \varphi_1 \, dx \qquad \text{and} \qquad  \lim_{n \to \infty} \into h_2 \bv_n \, dx  = \into h_1 \psi_1 \, dx.$$
Now, let us observe that $u_n(x)=\bu_n(x)r_n^{\frac{1}{p}}$, where $r_n^{\frac{1}{p}} \to + \infty$ and $\bu_n(x) \to \varphi_1(x)$ a.e. $x \in \Omega$. Since $\varphi_1 > 0$ in $ \Omega$, we deduce $u_n(x) \to + \infty$ a.e. $x \in \Omega$, and similarly we have $v_n(x) \to + \infty$ a.e. $x \in \Omega$. Hence, by assumption $\textbf{(F2)}$ we get
$$ \displaystyle F_s(x,u_n(x),v_n(x)) \to F_s^{++}(x) \qquad \text{a.e. } x \in \Omega,$$
and
$$ \displaystyle F_t(x,u_n(x),v_n(x)) \to F_t^{++}(x) \qquad \text{a.e. } x \in \Omega,$$
so by Lebesgue theorem we deduce
$$ \displaystyle \lim_{n \to \infty} \into F_s(x,u_n,v_n)\bu_n \, dx = \into F_s^{++} \varphi_1 \, dx $$
and
$$ \lim_{n \to \infty}  \into F_t(x,u_n,v_n) \bv_n \, dx = \into F_t^{++} \psi_1 \, dx. $$
In a similar way, by using also Fubini-Tonelli's Theorem we derive
\begin{equation*}
\displaystyle \lim_{n \to \infty} \into \int_0^1 F_s(x, \tau u_n, \tau v_n)\bu_n \, d \tau \, dx = \into F_s^{++} \varphi_1 \, dx
\end{equation*}
and
\begin{equation*}
\displaystyle \lim_{n \to \infty} \into \int_0^1 F_t(x, \tau u_n,\tau v_n)\bv_n \, d \tau \, dx = \into F_t^{++} \psi_1 \, dx.
\end{equation*}
Considering these convergences in \eqref{calcolo}, and denoting with $o(1)$  any quantity that goes to zero as $n$ goes to infinity, we have proved
\begin{align}\label{computation2}
\displaystyle
\left(1 +  \varepsilon_n \lVert z_n \rVert \right) C 
\geq & \left| r_n^{\frac{1}{p}} \left( 1 - \frac{1}{p}  \right) \left[ \into h_1 \varphi_1 \, dx - \into F_s^{++} \varphi_1 \, dx + o(1) \right] \right. \\
& \left. + r_n^{\frac{1}{q}}\left( 1 - \frac{1}{q}  \right) \left[ \into h_2 \psi_1 \, dx  - \into F_t^{++} \psi_1 \, dx + o(1) \right] \right|. \nonumber
\end{align}
If $p<q$, dividing  by $r_n^{\frac{1}{p}} \left(1 - \frac{1}{p} \right)$ and since $\varepsilon_n \to 0$ as $n \to \infty$, passing to limit we get
$$ \displaystyle 0 \geq \left|  \into h_1 \varphi_1 \, dx - \into F_s^{++} \varphi_1 \, dx    \right|,$$
that is a contradiction with both \eqref{LLpminoreq1} and \eqref{LLpminoreq2}.\\
If $p=q$, in a similar way we get 
$$ \displaystyle 0 \geq \left| \into h_1 \varphi_1 \,  + h_2 \psi_1 \, dx  - \into  F_s^{++} \varphi_1 + F_t^{++} \psi_1 \, dx   \right|,$$
that is a contradiction with both \eqref{LLpp1} and \eqref{LLpp2}.\\
If $p > q$, dividing  by $r_n^{\frac{1}{q}} \left(1 - \frac{1}{q} \right)$ and passing to limit we obtain 
$$ \displaystyle 0 \geq \left|  \into h_2 \psi_1 \, dx - \into F_t^{++} \psi_1 \, dx    \right|,$$\\
that is a contradiction with both \eqref{LLpmaggioreq1} and \eqref{LLpmaggioreq2}.\\
Let us recall that we assumed $(\bu,\bv)=(\varphi_1,\psi_1)$.

If it would be  $(\bu,\bv)= (-\varphi_1,-\psi_1)$, we deduce inequality \eqref{computation2} in which $F_s^{++}$ and $F_t^{++}$ are replaced by $F_s^{--}$ and $F_t^{--}$ respectively, and reasoning as before we get a contradiction under the same assumptions.

If now we consider $(\bu,\bv)= (\varphi_1,-\psi_1)$, convergences in \eqref{calcolo} give
\begin{align*}
\displaystyle
\left(1 +  \varepsilon_n \lVert z_n \rVert \right) C 
\geq & \left| r_n^{\frac{1}{p}} \left( 1 - \frac{1}{p}  \right) \left[ \into h_1 \varphi_1 \, dx - \into F_s^{+-} \varphi_1 \, dx + o(1) \right] \right. \\
& \left. + r_n^{\frac{1}{q}}\left( 1 - \frac{1}{q}  \right) \left[  \into F_t^{+-} \psi_1 \, dx -\into h_2 \psi_1 \, dx  + o(1) \right] \right|, \nonumber
\end{align*}
by which we infer a contradiction with \eqref{LLpminoreq1} and \eqref{LLpminoreq2} if $p<q$, with \eqref{LLpp1} and \eqref{LLpp2} if $p=q$, and with \eqref{LLpmaggioreq1} and \eqref{LLpmaggioreq2} if $p>q$.

If it would be  $(\bu,\bv)= (-\varphi_1,\psi_1)$, we deduce previous inequality in which $F_s^{+-}$ and $F_t^{+-}$ are replaced by $F_s^{-+}$ and $F_t^{-+}$ respectively, and we reach a contradiction under the same hypothesis.

We have so proved that the sequence $\{z_n\}_n=\{(u_n,v_n)\}_n$ satisfying \eqref{PS1} and \eqref{PS2} is bounded in $X$. This implies that there exists $z=(u,v) \in X$ such that, up to subsequence, $z_n$ converges to $z$ weakly in $X$ and strongly in $L^p(\Omega)\times L^q(\Omega)$.\\
In particular, as $n \to \infty$ we have
\begin{equation*}
\bigl\langle J'(z_n),(u_n-u,0)\bigr\rangle \rightarrow 0
\end{equation*}
and
\begin{equation*}
\bigl\langle J'(z_n),(0,v_n-v)\bigr\rangle \rightarrow 0.
\end{equation*}
Arguing as previously done to show the strong convergence of $\bu_n$ to $\bu$ in $W_0^{1,p}(\Omega)$ and the strong convergence of $\bv_n$ to $\bv$ in $W_0^{1,q}(\Omega)$, we conclude that up to subsequence $z_n \to z$ strongly in $X$.

\end{proof}

\section{GEOMETRY OF $J$}

\begin{proposition}
Let $F:\Omega \times \mathbb{R}^2 \to \mathbb{R}$ be a $C^1$-Carathéodory function satisfying \textbf{(F1)} and \textbf{(F2)}. Consider $h_1 \in L^{p'}(\Omega)$ and  $h_2 \in L^{q'}(\Omega)$. If we assume that

\begin{itemize}

\item[•]  in the case $p<q$,  condition \eqref{LLpminoreq1}   holds true;

\item[•]  in the case $p=q$,  condition \eqref{LLpp1}         holds true;

\item[•]  in the case $p>q$,  condition \eqref{LLpmaggioreq1} holds true;

\end{itemize}
then $J$ is coercive, i.e., 
$$ \displaystyle \lim_{\lVert(u,v) \rVert \to + \infty} J(u,v)=+ \infty.$$
\end{proposition}

\begin{proof} Arguing by contradicition, suppose that there exists a constant $c$ and a sequence $\{(u_n,v_n)\}_n$ in $X$ such that
\begin{equation}\label{coercconto1}
\displaystyle J(u_n,v_n) \leq c \qquad \text{for any } n \in \mathbb{N},
\end{equation} 
and
$$ \displaystyle \lVert (u_n,v_n) \rVert \to + \infty \qquad \text{as } n \to + \infty.$$
Let us set
$$ \displaystyle r_n:=\lVert u_n \rVert_{1,p}^p + \lVert v_n \rVert_{1,q}^q, \qquad \bu_n:= \frac{u_n}{r_n^{\frac{1}{p}}} \qquad \text{and} \qquad  \bv_n:=\frac{v_n}{r_n^{\frac{1}{q}}},$$
and observe that $\bz_n:=(\bu_n,\bv_n)$ satisfies 
$$ \displaystyle \lVert \bu_n \rVert_{1,p}^p + \lVert \bv_n \rVert_{1,q}^q = 1 \qquad \text{ for any } n \in \mathbb{N}.$$
Hence $ \lVert \bz_n \rVert \leq 2$ for any $n \in \mathbb{N}$ and there exists a subsequence, still denoted by $\lbrace{\bz_n\rbrace}_n,$ that converges to some $\bz=(\bu,\bv)$ weakly in $X$ and strongly in $L^p(\Omega)\times L^q(\Omega)$.\\
Dividing \eqref{coercconto1} by $r_n$, arguing as in the proof of (PS) condition and considering also the characterization of $\lambda_1$ and the weak lower semicontinuity of both $\lVert \cdot \lVert_{1,p}^p$ and $\lVert \cdot \lVert_{1,q}^q$, we get
\begin{align*}
\displaystyle 
\Phi(\bu,\bv)  \leq \liminf_{n \to \infty} \Phi(\bu_n,\bv_n) \leq \limsup_{n \to \infty} \Phi(\bu_n,\bv_n)   \leq \lambda_1 \Psi(\bu,\bv) \leq \Phi(\bu,\bv).
\end{align*}
In particular, since $\lVert \bu_n \rVert_{1,p}^p + \lVert \bv_n \rVert_{1,q}^q =1 $ for any $n \in \mathbb{N},$ we deduce that $(\bu,\bv)$ is a nontrivial eigenfunction associated to $\lambda_1$, hence there exists $\theta>0$ such that only one of the following possibilities occurs:
$$ \displaystyle (\bu,\bv)= (\theta^{\frac{1}{p}}\varphi_1,\theta^{\frac{1}{q}} \psi_1), \qquad (\bu,\bv)= (-\theta^{\frac{1}{p}}\varphi_1,-\theta^{\frac{1}{q}}\psi_1),$$
$$ (\bu,\bv)= (\theta^{\frac{1}{p}}\varphi_1,-\theta^{\frac{1}{q}}\psi_1), \qquad (\bu,\bv)= (-\theta^{\frac{1}{p}}\varphi_1,\theta^{\frac{1}{q}}\psi_1).$$
Let us assume $(\bu,\bv)= (\theta^{\frac{1}{p}}\varphi_1,\theta^{\frac{1}{q}} \psi_1)$.\\
By the mean value theorem, the characterization of $\lambda_1$ and the definition of the normalized sequence $\bz_n=(\bu_n,\bv_n)$, we infer that
\begin{align*}
\displaystyle
c 
& \geq J(u_n,v_n) \\
& \geq - \into F(x,u_n,v_n) \, dx + \into h_1 u_n \, dx + \into h_2 v_n \, dx \nonumber \\
& = r_n^{\frac{1}{p}} \left[  \into h_1 \bu_n \, dx  \ -  \into  \int_0^1 F_s(x,\tau u_n,\tau v_n)\bu_n \, d \tau \, dx \right] \nonumber \\
&+ r_n^{\frac{1}{q}} \left[ \into h_2 \bv_n \, dx  \ -  \into  \int_0^1 F_t(x,\tau u_n,\tau v_n)\bv_n \, d \tau \, dx   \right] - \into F(x,0,0) \, dx. \nonumber
\end{align*}
Since $(\bu_n,\bv_n) \to (\theta^{\frac{1}{p}} \varphi_1, \theta^{\frac{1}{q}} \psi_1)$  strongly in $L^p(\Omega)\times L^q(\Omega)$, we have
\begin{align*}
& \displaystyle  \lim_{n \to \infty}\into h_1 \bu_n \, dx = \theta^{\frac{1}{p}} \into h_1 \varphi_1 \, dx,\\ &\displaystyle \lim_{n \to \infty} \int_0^1 F_s(x, \tau u_n, \tau v_n)\bu_n \, d \tau \, dx = \theta^{\frac{1}{p}} \into F_s^{++} \varphi_1 \, dx,
\end{align*}
and
\begin{align*}
& \displaystyle \lim_{n \to \infty} \into h_2 \bv_n \, dx  = \theta^{\frac{1}{q}} \into h_1 \psi_1 \, dx,\\
& \displaystyle \lim_{n \to \infty} \into \int_0^1 F_t(x, \tau u_n,\tau v_n)\bv_n \, d \tau \, dx = \theta^{\frac{1}{q}} \into F_t^{++} \psi_1 \, dx.
\end{align*}
Therefore, considering these limits in the previous inequality we deduce that
\begin{align*}
\displaystyle
c 
& \geq \theta^{\frac{1}{p}}r_n^{\frac{1}{p}} \left[  \into h_1 \varphi_1 \, dx  \ -  \into  F_s^{++} \varphi_1 \, dx + o(1) \right] \\
&+ \theta^{\frac{1}{q}}r_n^{\frac{1}{q}} \left[ \into h_2 \psi_1 \, dx  \ -  \into F_t^{++} \psi_1 \, dx  + o(1) \right]  - \into F(x,0,0) \, dx.
\end{align*}
In the case $p<q$, dividing by  $\theta^{\frac{1}{p}}r_n^{\frac{1}{p}}$ and passing to limit as $n \to \infty$, we have
$$ \displaystyle 0 \geq \into h_1 \varphi_1 \, dx  \ -  \into  F_s^{++} \varphi_1 \, dx,$$
which contradicts \eqref{LLpminoreq1}.\\
When $p=q$, similarly we obtain
$$ \displaystyle 0 \geq \into h_1 \varphi_1 \, dx + h_2 \psi_1 \, dx  -  \into  F_s^{++} \varphi_1 + F_t^{++} \psi_1  \, dx,$$ 
contradicting \eqref{LLpp1}.\\
If $p>q$, dividing by  $\theta^{\frac{1}{q}}r_n^{\frac{1}{q}}$ and passing to limit as $n \to  \infty$ we get
$$ \displaystyle 0 \geq \into h_2 \psi_1\, dx  \ -  \into  F_t^{++} \psi_1 \, dx,$$ 
contradicting \eqref{LLpmaggioreq1}.\\
These contradiction show that $(\bu,\bv)$ cannot be equal to $(\theta^{\frac{1}{p}}\varphi_1,\theta^{\frac{1}{q}}\psi_1).$

In a similar way if it would be $(\bu,\bv)=(-\theta^{\frac{1}{p}}\varphi_1,-\theta^{\frac{1}{q}}\psi_1)$, one obtains a contradiction with the hypotheses \eqref{LLpminoreq1}, \eqref{LLpp1} and \eqref{LLpmaggioreq1}, involving $F_s^{--}$ and $F_t^{--}$.

If now we assume $(\bu,\bv)=(\theta^{\frac{1}{p}}\varphi_1,-\theta^{\frac{1}{q}}\psi_1),$ we obtain
\begin{align*}
\displaystyle
c 
& \geq \theta^{\frac{1}{p}}r_n^{\frac{1}{p}} \left[  \into h_1 \varphi_1 \, dx  \ -  \into  F_s^{+-} \varphi_1 \, dx + o(1) \right] \\
&+ \theta^{\frac{1}{q}}r_n^{\frac{1}{q}} \left[  \into F_t^{+-} \psi_1 \, dx - \into h_2 \psi_1 \, dx  + o(1) \right]  - \into F(x,0,0) \, dx,
\end{align*}
and arguing as before we get a contradiction with  \eqref{LLpminoreq1}, \eqref{LLpp1} and \eqref{LLpmaggioreq1}.

Similarly, if it would be $(\bu,\bv)=(-\theta^{\frac{1}{p}}\varphi_1,\theta^{\frac{1}{q}}\psi_1)$, one obtains a contradiction with the hypotheses \eqref{LLpminoreq1}, \eqref{LLpp1} and \eqref{LLpmaggioreq1}, involving $F_s^{-+}$ and $F_t^{-+}$.
\end{proof}
\noindent
{\mbox {\it Proof of Theorem~\ref{mainresult}.~}} Let us consider assumption  \eqref{LLpminoreq1} if $p<q$, \eqref{LLpp1} if $p=q$ and \eqref{LLpmaggioreq1} if $p>q$. In these cases we have proved that the functional $J$ is coercive. This and the sequentially weakly lower semicontinuity of $J$ on the reflexive Banach space $X$ imply by Weierstrass Theorem  that $J$ has a global minimum point, i.e. there exists a weak solution for problem \eqref{Syst0}.\\
Let us now consider assumption  \eqref{LLpminoreq2} if $p<q$, \eqref{LLpp2} if $p=q$ and \eqref{LLpmaggioreq2} if $p>q$.\\
We will show that
\begin{align}\label{limiti}
\displaystyle
& \lim_{\theta \to + \infty} J(\theta^{\frac{1}{p}} \varphi_1,\theta^{\frac{1}{q}} \psi_1)=-\infty, & \lim_{\theta \to + \infty} J(-\theta^{\frac{1}{p}} \varphi_1,-\theta^{\frac{1}{q}} \psi_1)=-\infty, \nonumber \\
&\\
& \lim_{\theta \to + \infty} J(\theta^{\frac{1}{p}} \varphi_1,-\theta^{\frac{1}{q}} \psi_1)=-\infty, & \lim_{\theta \to + \infty} J(-\theta^{\frac{1}{p}} \varphi_1,\theta^{\frac{1}{q}} \psi_1)=-\infty.  \nonumber 
\end{align}
By characterization of $\lambda_1$ and the mean value formula, for any $\theta >0$ we infer that
\begin{align*}
\displaystyle
J(\theta^{\frac{1}{p}} \varphi_1,\theta^{\frac{1}{q}} \psi_1)
& = \theta^{\frac{1}{p}} \into h_1 \varphi_1 \, dx + \theta^{\frac{1}{q}} \into h_2 \psi_1 \, dx - \into F(x,\theta^{\frac{1}{p}} \varphi_1,\theta^{\frac{1}{q}} \psi_1) \, dx \\
& = \theta^{\frac{1}{p}} \left[ \into h_1 \varphi_1 \, dx  -  \into \int_0^1 F_s(x,\tau \theta^{\frac{1}{p}} \varphi_1,\tau \theta^{\frac{1}{q}} \psi_1) \varphi_1 \, d \tau \, dx \right] \\
& + \theta^{\frac{1}{q}} \left[ \into h_2 \psi_1 \, dx  -  \into \int_0^1 F_t(x,\tau \theta^{\frac{1}{p}} \varphi_1,\tau \theta^{\frac{1}{q}} \psi_1) \psi_1 \, d \tau \, dx \right] \\
& - \into F(x,0,0) \, dx.
\end{align*}
By \textbf{(F2)}, Fubini-Tonelli and Lebesgue theorems, it follows
$$ \displaystyle \lim_{\theta \to + \infty} \into \int_0^1 F_s(x,\tau \theta^{\frac{1}{p}} \varphi_1,\tau \theta^{\frac{1}{q}} \psi_1) \varphi_1 \, d \tau \, dx = \into F_s^{++} \varphi_1  \, dx;$$
$$ \displaystyle \lim_{\theta \to + \infty} \into \int_0^1 F_t(x,\tau \theta^{\frac{1}{p}} \psi_1,\tau \theta^{\frac{1}{q}} \psi_1) \psi_1 \, d \tau \, dx = \into F_t^{++} \psi_1  \, dx.$$
Therefore,
\begin{align*}
\displaystyle
J(\theta^{\frac{1}{p}} \varphi_1,\theta^{\frac{1}{q}} \psi_1)
& = \theta^{\frac{1}{p}} \left[ \into h_1 \varphi_1 \, dx  - \into F_s^{++} \varphi_1  \, dx  + o(1)\right] \\
& + \theta^{\frac{1}{q}} \left[ \into h_2 \psi_1 \, dx  -  \into F_t^{++} \psi_1  \, dx + o(1) \right] \\
& - \into F(x,0,0) \, dx,
\end{align*}
by which
\begin{equation*}
\displaystyle 
\lim_{\theta \to + \infty} J(\theta^{\frac{1}{p}} \varphi_1,\theta^{\frac{1}{q}} \psi_1)=- \infty.
\end{equation*}
Arguing in a similar way, considering the other inequalities in assumptions  \eqref{LLpminoreq2}, \eqref{LLpp2} and \eqref{LLpmaggioreq2} involving $F_s^{--}$ and $F_t^{--}$,  $F_s^{+-}$ and $F_t^{+-}$, and $F_s^{-+}$ and $F_t^{-+}$ respectively, we get the other limits in \eqref{limiti}.

Let $\lambda_2$ be the eigenvalue of \eqref{EigenDeThélin0} defined in \eqref{lambdakSfera} and let us  set
$$\Lambda_2:=\{ (u,v) \in X \; : \; \Phi(u,v) \geq \lambda_2 \Psi(u,v)     \}.$$
By $\textbf{(F1)}$ and the mean value formula, we have
$$ \displaystyle \left| F(x,s,t) \right| \leq M|s| + M|t| + \left|F(x,0,0) \right|.$$
Denoting with $\lambda_{1,p}$ and $\lambda_{1,q}$ the first eigenvalue on $\Omega$ under homogeneous Dirichlet boundary condition of $\Delta_p$ and $\Delta_q$ respectively, by H\"older and Poincaré inequality we infer that
\begin{align*}
\displaystyle
& \left|  - \into F(x,u,v) \, dx + \into h_1 u \, dx + \into h_2 v \, dx \right| \\
\leq & \into  \left| F(x,0,0) \right|  \, dx + M \into \left| u  \right| \, dx +  M \into \left| v  \right| \, dx +  \into \left| h_1 \right| \left| u  \right| \, dx + \into \left| h_2 \right| \left| v \right| \, dx \nonumber \\
\leq & \into \left|F(x,0,0)\right| \, dx + \left( M \left| \Omega \right|^{\frac{1}{p'}} +  \lVert h_1 \rVert_{p'}  \right) \lVert u \rVert_p +  \left( M \left| \Omega \right|^{\frac{1}{q'}} + \lVert h_2 \rVert_{q'} \right) \lVert v \rVert_q \nonumber \\
\leq &  \into \left|F(x,0,0)\right| \, dx + \frac{\left( M \left| \Omega \right|^{\frac{1}{p'}} +  \lVert h_1 \rVert_{p'}  \right)}{ \lambda^{\frac{1}{p}}_{1,p}} \lVert u \rVert_{1,p} +  \frac{\left( M \left| \Omega \right|^{\frac{1}{q'}} + \lVert h_2 \rVert_{q'} \right)}{\lambda^{\frac{1}{q}}_{1,q}} \lVert v \rVert_{1,q},\nonumber
\end{align*}
for any $(u,v) \in X$.\\
On the other hand, for any $(u,v) \in \Lambda_2$ we get
\begin{align*}
\displaystyle \Phi(u,v) - \lambda_1 \Psi(u,v) 
& \geq \left( 1 - \frac{\lambda_1}{\lambda_2}\right)  \Phi(u,v)  \\
& = \frac{\alpha+1}{p}\left( 1 - \frac{\lambda_1}{\lambda_2}\right)\lVert u \rVert^p_{1,p}  + \frac{\beta+1}{q}\left( 1 - \frac{\lambda_1}{\lambda_2}\right)\lVert v \rVert^q_{1,q}. \nonumber
\end{align*}
Altogether, for any $(u,v) \in \Lambda_2$ we have
\begin{align*}
\displaystyle
J(u,v)
& \geq \frac{\alpha+1}{p}\left( 1 - \frac{\lambda_1}{\lambda_2}\right)\lVert u \rVert^p_{1,p}  - \frac{\left( M \left| \Omega \right|^{\frac{1}{p'}} +  \lVert h_1 \rVert_{p'}  \right)}{ \lambda^{\frac{1}{p}}_{1,p}} \lVert u \rVert_{1,p} \\
& + \frac{\beta+1}{q}\left( 1 - \frac{\lambda_1}{\lambda_2}\right)\lVert v \rVert^q_{1,q} - \frac{\left( M \left| \Omega \right|^{\frac{1}{q'}} + \lVert h_2 \rVert_{q'} \right)}{\lambda^{\frac{1}{q}}_{1,q}} \lVert v \rVert_{1,q}\\
& - \into \left|F(x,0,0)\right| \, dx.
\end{align*}
Since $\lambda_1$ is isolated, we have $\lambda_2 > \lambda_1$ and so we deduce that $J$ is coercive on $\Lambda_2$, and in particular it is bounded from below.
Considering also \eqref{limiti}, we have shown that there exists $\Theta >0$ such that, setting
$$ \displaystyle E_1^{\Theta}= \bigg\{ ( \left|  \theta \right|^{\frac{1}{p}} \varphi_1 ,  \left|  \theta \right|^{\frac{1}{q}} \psi_1 )  \text{sgn} (\theta)  : \left|\theta \right| \geq \Theta \bigg\}   \bigcup  \bigg\{ (- \left|  \theta \right|^{\frac{1}{p}} \varphi_1 ,  \left|  \theta \right|^{\frac{1}{q}} \psi_1 )  \text{sgn} (\theta)  :  \left|\theta \right| \geq \Theta \bigg\},$$ 
we have
$$ \displaystyle \gamma:= \sup_{(u,v) \in E_1^{\Theta}} J(u,v) < \inf_{(u,v) \in \Lambda_2} J(u,v):= \delta.$$
Consider the family of mappings 
$$ 
\Gamma:= \{ h \in C( [-1,1] ,X) \, : \, h(1)=-h(-1)\in E_1^{\Theta}  \}.
$$
Notice that $\displaystyle \tilde{h}(\tau):= ( \tau \Theta^{\frac{1}{p}} \varphi_1 , \tau \Theta^{\frac{1}{q}} \psi_1 )$,  $\tau \in [-1,1]$, belongs to $\Gamma$ and thus $\Gamma\not=\emptyset$.\\
Let us show that $h(B_1) \cap \Lambda_2 \neq \emptyset$ for every $h \in \Gamma$, where $B_1:=[-1,1]$.\\ Indeed, this is immediate if $(0,0) \in h(B_1)$. When $(0,0) \not \in h(B_1)$, we consider the odd map $\pi: X \setminus (0,0) \to \mathcal{M}$ given by 
$$ \displaystyle \pi(u,v):= \left( \frac{u}{\Phi(u,v)^{\frac{1}{p}}}   ,  \frac{v}{\Phi(u,v)^{\frac{1}{q}}}     \right), \quad (u,v)\in X\setminus (0,0),$$
and the map $\alpha_0:S^1 \to \mathcal{M}$ defined for every $(\xi_1,\xi_2)\in S^1$ by
$$ 
\displaystyle
\alpha_0(\xi_1,\xi_2):=
\begin{cases}
\; \; \, \pi ( \, h(\xi_1) \, ) & \text{ if } \xi_2 \geq 0, \bigskip \\
- \pi  ( \, h (-\xi_1) \, ) & \text{ if } \xi_2 < 0.
\end{cases}
$$
Observe that by the definition of $\alpha_0$, we have $\alpha_0(-\xi_1,-\xi_2) = -\alpha_0(\xi_1,\xi_2)$ for every $(\xi_1,\xi_2)\in S^1$ with $\xi_2\not=0$. 
On the other hand, 
since $h(1)=-h(-1)$, the oddness of $\pi$ implies that  $\alpha_0(-1,0) = -\alpha_0(1,0)$ and, consequently,  $\alpha_0$ is odd in all $S^1$. In addition, using again that $h(1)=-h(-1)$, the map $\alpha_0$ is also continuous. In particular, $\alpha_0(S^{1})\in \mathcal{M}_2$, where  $\mathcal{M}_2$ is defined in Proposition \ref{prop 2.9}.\\
Recalling that
$$ \displaystyle \lambda_2:= \frac{1}{c_2} \qquad \text{where} \quad c_2:= \sup_{A \in \mathcal{M}_2} \min_{(u,v) \in A} Q_{ \mathcal{M}}(u,v)=\sup_{A \in \mathcal{M}_2} \min_{(u,v) \in A} \frac{\Psi(u,v)}{\Phi(u,v)},$$
we have
$$
\min_{(u,v) \in\alpha_0(S^1)} \frac{\Psi(u,v)}{\Phi(u,v)} \leq c_2.
$$
Let $(u_0,v_0) \in \alpha_0 (S^1)$ be such that 
 $ \displaystyle \frac{\Psi(u_0,v_0) }{\Phi (u_0,v_0)} \leq c_2$, i.e.,
 $ \displaystyle \Phi(u_0,v_0)\geq \lambda_2 \Psi(u_0,v_0)$ and 
 so $(u_0,v_0) \in \alpha_0(S^1) \cap \Lambda_2$. 
By definition of $\alpha_0$, this implies that
 there exists $\tau \in [-1,1]$ such that $\pi(h(\tau)) \in \Lambda_2$. 
 Now we observe that the $(p,q)$-homogeneity of the functionals 
 $\Phi$ and $\Psi$ implies that  $h(\tau) \in \Lambda_2$ 
and thus $h(B_1) \cap \Lambda_2 \neq \emptyset$.\\
In particular, 
$$ 
\displaystyle  c:= \inf_{h \in \Gamma}  \max_{\tau \in B_1} J (h(\tau)) \geq \delta > \gamma
$$
Now, we  show that 
$c$
is a critical value of $J$.\\
Indeed, let us assume by contradiction that $c$ is a regular value of $J$. 
Fixing $0 < \bar{\varepsilon} < c - \gamma,$ by \cite[Theorem 3.4]{STRUWE} there 
exists $ \varepsilon \in (0, \bar{\varepsilon})$ and a family of homeomorphisms $ \phi: X \times [0,1] \to X$ satisfying the following properties:
\begin{enumerate}
\item[$(i)$] $ \phi ( z,\tau)=z$ if $\tau=0$ or if $ \left|  J(z) - c \right| \geq \bar{\varepsilon} $;
\item[$(ii)$]$ J(\phi ( z,\tau))$ is non-increasing in $\tau$ for any $ z \in X$;
\item[$(iii)$] If $J(z) \leq c + \varepsilon$, then $J( \phi(z,1)) \leq c - \varepsilon$.
\end{enumerate} 
For any $z \in E_1^{\Theta}$, we have $J(z) \leq \gamma < c - \bar{\varepsilon}$ 
and so by property $(i)$ we infer that $\phi$ leaves the set $E_1^{\Theta}$ fixed. 
By definition of $c$ there exists $h \in \Gamma$ such that 
$ \max_{\tau \in [-1,1]} J (h(\tau)) < c + \varepsilon$. Let us define 
$\hat{h}(\cdot):= \phi( h(\cdot) , 1).$ 
Since $h(\pm 1)\in E_1^{\Theta}$, 
we get  $\hat{h}(\pm 1)= \phi( h(\pm 1) , 1) = h(\pm 1)$, from which we deduce 
that $\hat{h}(-1)= -\hat{h}(1)$ and hence
 $\hat{h} \in \Gamma$. Now,  
property $(iii)$ gives
$ \max_{\tau \in [-1,1]} J (\hat{h}(\tau)) \leq c - \varepsilon$, 
which  contradicts the definition of $c$.

\qed

\noindent
{\bf Acknowledgments.}\\ 
The first author is supported by Junta de Andalucía (grant FQM-116).
The second and third authors are supported by INdAM-GNAMPA, and thank PNRR MUR project CN00000013 HUB - National Centre for HPC, Big Data and Quantum Computing (CUP H93C22000450007).

\end{document}